\documentclass{article}
\usepackage{amsmath,amssymb,amsfonts,amsthm,latexsym,amstext,amscd}
\usepackage{epsfig,graphics,graphicx,color}
\usepackage{enumerate}
\usepackage{hyperref}
\usepackage{mathdots} %
\usepackage{tikz}
\usetikzlibrary{patterns}
\newcommand{\notiz}[1]{}
\newcommand{\bsx}{\boldsymbol{x}}%
\newcommand{\bsu}{\boldsymbol{u}}%
\newcommand{\bsl}{\boldsymbol{l}}%
\newcommand{\bse}{\boldsymbol{e}}%
\newcommand{\bszero}{\boldsymbol{0}}%

\newcommand{\NN}{\mathbb N}
\newcommand{\ZZ}{\mathbb Z}
\newcommand{\RR}{\mathbb R}
\newcommand{\QQ}{\mathbb Q}
\newcommand{\PP}{\mathbb P}

\newcommand{\Bad}{{\rm Bad}}

\newcommand{\lc}{{\rm lc}}
\newcommand{\FF}{\mathbb{F}}

\newcommand{\Fq}{\mathbb{F}_q}

\newtheorem{theorem}{Theorem}
\newtheorem{RQ}{Research Question}
\newtheorem{lemma}{Lemma}
\newtheorem{remark}{Remark}
\newtheorem{example}{Example}

\newtheorem{coro}{Corollary}
\newtheorem{defi}{Definition}

\newtheorem{algo}{Algorithm}

\begin{document}

\title{Hankel Matrix Decompositions and Digital Sequences over Function Fields}

\author{Roswitha Hofer\thanks{Institute of Financial Mathematics and Applied Number Theory, Johannes Kepler University Linz, Altenbergerstr. 69, 4040 Linz, Austria. e-mail: roswitha.hofer@jku.at}}

\date{\today}
\maketitle
\begin{abstract}
	This paper studies formal Laurent series, with particular emphasis on their continued-fraction and Laurent-series expansions. Knowledge of both expansions is important, for example, in the study of Kronecker-type sequences in uniform distribution theory and in the search for potential counterexamples to variants of Littlewood's conjecture over fields of formal Laurent series.
	
	Our approach investigates Hankel matrices over a field and relates them to Laurent-series expansions over the same field. We describe such matrices via a block $\mathcal{L}\mathcal{B}\mathcal{U}$ decomposition, using recursive algorithms based on the continued-fraction coefficients of the corresponding Laurent series. The upper-triangular factor $\mathcal{U}$ has several remarkable properties. In particular, it is constructed columnwise from the so-called Ostrowski-type representations of powers of $X$. Applying $\mathcal{U}$ in Niederreiter's digital construction yields an Ostrowski-type van der Corput sequence that can be related to the corresponding Kronecker-type sequence. For ordinary Kronecker sequences, such a relation is known in only one case: the sequence $(n\varphi)_{n\geq 0}$, where $\varphi$ denotes the golden ratio, and the associated Zeckendorf van der Corput sequence. Drmota established a relation between the distribution of the latter sequence and that of the $(n\varphi)$-sequence.
	
	The original motivation for our study of Hankel matrices was to recast the search for counterexamples to variants of Littlewood's conjecture over fields of Laurent series as the search for specific Kronecker-type sequences. We hope that the digital-sequence viewpoint will lead to new approaches to, and perspectives on, the search for such counterexamples. This paper also discusses several open problems concerning specific potential counterexamples, provided that they exist. Interestingly, the method developed for the study of Kronecker-type sequences also allows a new perspective on the celebrated Sobol' construction. Our approach yields a short and elegant proof of its excellent distribution properties. Sobol' sequences are among the most popular classes of so-called low-discrepancy sequences and are frequently used in quasi-Monte Carlo (QMC) methods.
\end{abstract}

\noindent{\textbf{Keywords:} 
 formal Laurent series, simple continued fraction, convergents,  block $\mathcal{L}\mathcal{B}\mathcal{U}$-decomposition of Hankel matrices, Ostrowski-type representation, Kronecker-type sequences, variants of the Littlewood conjecture, Sobol' sequences}\\
\noindent{\textbf{MSC2010:} primary 11K38; secondary 11J70, 15B05, 11J61, 65C05.}

\section{Introduction and Motivation}
We recall that every $\alpha\in\RR$ can be represented by its simple continued fraction
$$\alpha=[\alpha]+\cfrac{1}{\alpha_1 + \cfrac{1}{\alpha_2 + \cfrac{1}{\alpha_3 + \dots}}},$$
where $[\alpha]$ denotes the integer part of $\alpha$ and $\alpha_1,\alpha_2,\ldots\in\NN$. It is well known that $\alpha$ is badly approximable, i.e.,
$$\alpha\in\{\beta \in\RR\,:\, \inf_{h\in\ZZ\setminus\{0\}} |h|\cdot\|h\beta\|>0\},$$
where $\|x\|$ denotes the distance from $x$ to the nearest integer, if and only if the coefficients $\alpha_1,\alpha_2,\ldots\in\NN$ in its simple continued-fraction expansion are bounded.\\

In this manuscript, for several reasons, we focus on badly approximable Laurent series $\theta\in\FF((X^{-1}))$. 
In analogy with the number field $\QQ$, we use the rational function field $\FF(X)$, where $\FF$ is an arbitrary field, such as $\QQ$, $\RR$, or a finite field $\Fq$ with $q$ elements. The analogue of the absolute value on $\RR$ is the non-archimedean absolute value $|\,\cdot\,|$, which is defined in terms of the degree valuation $\nu$ as follows. 

For every $f\in\FF(X)$, there exist $p(X),\,q(X)\in\FF[X]$ with $q(X)\neq 0$ such that $f=\frac{p(X)}{q(X)}$. We write $q(X)=\sum_{i=0}^ra_iX^i$, where $a_i\in\FF$ and $a_r\neq 0$. Then the degree $\deg(q(X))$ of $q(X)$ is, by definition, $r$. We set $\deg(0):=-\infty$.  \\
The degree valuation $\nu$ of the rational function $f$ is then defined by
$$\nu(f):=\deg(p(X))-\deg(q(X)),$$
and the corresponding non-archimedean absolute value is given by
$$|f|:=e^{\nu(f)},$$
where $e$ denotes Euler's number.\\ 
The completion of $\FF(X)$ with respect to the absolute value $|\cdot|$ is the field of formal Laurent series, denoted by $\FF((X^{-1}))$. 
Every nonzero element $\theta(X)$ of $\FF((X^{-1}))$ can be written uniquely as
$$\sum_{i=-u}^\infty c_iX^{-i}=[\theta(X)]+\sum_{\substack{i=1\\ i\geq -u}}^\infty c_iX^{-i},$$
where $u\in\ZZ$ and $c_i\in\FF$ for $i\geq -u$, with $c_{-u}\neq 0$. Here,
$$[\theta(X)]=\sum_{i=-u}^{0}c_{i}X^{-i}$$
is the polynomial part of $\theta(X)$, which belongs to $\FF[X]$. Note that the absolute value of $|\theta(X)-[\theta(X)]|$ is smaller than $1$ and that this quantity gives the distance to the nearest polynomial. This distance will be denoted by $\|\cdot\|$ below. The element $\theta(X)-[\theta(X)]$ is often abbreviated to $\langle\theta(X)\rangle$ and called the fractional part of $\theta(X)$. 
The degree valuation $\nu(\theta(X))$ of $\theta(X)$ is $u$, and the absolute value $|\theta(X)|$ of $\theta(X)$ is $e^{u}$.

The set of all possible fractional parts, namely,
$$\{\theta(X)\in\FF((X^{-1})):|\theta(X)|<1\},$$ 
will be denoted by $\overline{\FF((X^{-1}))}$ below.

For a short discussion of simple continued fractions of $\alpha \in \RR$ and of $\theta(x)\in\FF((X^{-1}))$ we refer to the Appendix.

In analogy with badly approximable real numbers, we define badly approximable Laurent series as follows.
\begin{defi}[Badly Approximable Laurent series]\label{def:set_Bad}
	We define the set $\Bad$ of badly approximable Laurent series in $\FF((X^{-1}))$ by
	$$\Bad:=\{ \theta(X) \in\FF((X^{-1})): \inf_{Q(X)\in\FF[X]\setminus\{0\}}|Q(X)|\cdot\| Q(X)\theta(X)\|>0\}.$$
\end{defi}
Note that a Laurent series $\theta(X)$ has a simple continued-fraction expansion
$$\theta(X)=[\theta(X)]+\cfrac{1}{A_1(X) + \cfrac{1}{A_2(X) + \cfrac{1}{A_3(X) + \dots}}}$$
with $\sup_{i\in\NN}\deg A_i(X)<\infty$ if it is in $\Bad$, and conversely.

We are particularly interested in the formal power-series representations of badly approximable Laurent series. 
We would like to point out that, for real numbers, and in particular for badly approximable ones, it is a non-trivial problem to determine both their explicit $b$-adic representation for a fixed base $b$, such as $2$ or $10$, and their simple continued-fraction expansion. So far, there are only a few examples for which explicit information about both representations is available (see, e.g., Shallit \cite{shallit1,shallit2}).

\subsection{The Littlewood conjecture and its variants}

The question of the formal-series expansions of badly approximable Laurent series is relevant, for instance, to the search for potential counterexamples to variants of Littlewood's conjecture over $\FF((X^{-1}))$. These variants over a finite field $\FF_q$ are related to the distribution of so-called Kronecker-type and Kronecker--Halton-type sequences, which form special classes of digital sequences (see, e.g., \cite{Hofer2026} for a recent survey).

The famous Littlewood conjecture, dating back to the 1930s, states that, for $s>1$ and every $(\alpha_1,\ldots,\alpha_s)\in\RR^s$, one has
\begin{equation}\label{eq:Littlewood_real}
	\inf_{(h_1,\ldots,h_s)\in\ZZ^s\setminus\{\boldsymbol{0}\}} \max\{1,|h_1|\}\cdots\max\{1,|h_s|\}  \cdot\|h_1\alpha_1+\cdots+h_s\alpha_s\|=0.
\end{equation}

Its $p$-adic variant \cite{MatTeu2004} asserts that, for every $\alpha\in\RR$ and every prime $p\in\PP$,
\begin{equation}\label{eq:padicLittlewood_real}
	\inf_{h\in\ZZ\setminus\{0\},\,k\geq 0} |h|\,\cdot\,\|hp^k \alpha\|=0.
\end{equation}

Both conjectures \eqref{eq:Littlewood_real} and \eqref{eq:padicLittlewood_real} remain unresolved.

Reformulating these conjectures over a field of Laurent series $\FF((X^{-1}))$ rather than over $\RR$ leads to the following variants of Littlewood's conjecture. See \cite{DavLew1963} for \eqref{eq:Littlewood_series} and \cite{MatTeu2004} for \eqref{eq:padicLittlewood_series}. 

Let $s>1$. Is it true that, for every $(\theta_1,\ldots,\theta_s)\in\FF((X^{-1}))^s$,
\begin{equation}\label{eq:Littlewood_series}
	\inf_{(Q_1,\ldots,Q_s)\in\FF[X]^s\setminus\{\boldsymbol{0}\}} \max\{1,|Q_1|\}\cdots \max\{1,|Q_s|\} \cdot\|Q_1\theta_1+\cdots+Q_s\theta_s\|=0?
\end{equation}  
Is it true that, for every $\theta\in\FF((X^{-1}))$ and every monic irreducible polynomial $p(X)\in\FF[X]$,
\begin{equation}\label{eq:padicLittlewood_series}
	\inf_{Q\in\FF[X]\setminus\{0\},\,k\geq 0} |Q|\,\cdot\,\|Qp(X)^k \theta\|=0?
\end{equation}

Counterexamples to \eqref{eq:Littlewood_series} and \eqref{eq:padicLittlewood_series} are known when the field $\FF$ has characteristic $0$ (see, e.g., \cite{DavLew1963,Baker1964,MatTeu2004,BugMat2008}). In every positive characteristic, the variant \eqref{eq:Littlewood_series} of Littlewood's conjecture remains open.

Counterexamples to \eqref{eq:padicLittlewood_series} have recently been identified in positive odd characteristics. Consequently, the only unresolved case for the validity of the variant \eqref{eq:padicLittlewood_series} of Littlewood's conjecture is that in which the ground field $\FF$ has characteristic $2$. 
The first counterexample to \eqref{eq:padicLittlewood_series} over a finite field was identified in characteristic $3$ in \cite{Adiceametal}, followed by further counterexamples in low odd characteristics \cite{GR}. The restriction to small characteristics is due to the computer-assisted nature of the method. In \cite{AdBad}, it was proved that the counterexample in \cite{Adiceametal} remains a counterexample in every characteristic $p\equiv 3 \pmod{4}$. Finally, \cite{LaiSprang} introduced a comparatively elementary method for proving the existence of a concrete counterexample, generalizing the one in \cite{Adiceametal} and depending on the odd characteristic $p$.

It is worth noting that the proof in \cite{LaiSprang} starts from a simple closed-form expression for the Laurent-series expansion of the counterexamples. Since it immediately follows that all counterexamples to \eqref{eq:Littlewood_series} and \eqref{eq:padicLittlewood_series} are necessarily badly approximable, it is therefore natural to study the series expansions of badly approximable elements in $\FF((X^{-1}))$, in particular in the unresolved cases: \eqref{eq:padicLittlewood_series} in characteristic $2$ and \eqref{eq:Littlewood_series} in any positive characteristic.
	
\begin{remark}\label{rem:CE_BA}
	We emphasize that any potential counterexamples to the variants of Littlewood's conjecture in \eqref{eq:Littlewood_real}, \eqref{eq:padicLittlewood_real}, \eqref{eq:Littlewood_series}, and \eqref{eq:padicLittlewood_series} are also of interest in the theory of uniform distribution. In this context, one studies, for example, the distribution of Kronecker and Kronecker--Halton sequences, as well as their analogues over $\FF_q((X^{-1}))$, in the $s$-dimensional unit cube $[0,1)^s$. We refer the interested reader to, for example, \cite{Hofer2026} for a recent survey.
\end{remark}

\subsection{Hankel matrices, Kronecker-type sequences and digital sequences}

In this paper, we study Hankel matrices $\mathcal{M}_\theta$ defined by the coefficients in the Laurent-series expansion of a series $\theta\in\overline{\FF((X^{-1}))}$.

\begin{defi}[Hankel matrix associated with $\theta$]\label{def:Hankelmatrix}
	Let $\theta\in\overline{\FF((X^{-1}))}$ have Laurent-series expansion $\theta=\sum_{i=1}^\infty c_iX^{-i}$. 
	We define the Hankel matrix $\mathcal{M}_{\theta}=(m_{i,j})_{i,j\geq 0}\in\FF^{\NN_0\times\NN_0}$ associated with $\theta$ by setting $m_{i,j}:=c_{i+j+1}$ for all $i,j\in\NN_0$. 
\end{defi}

A basic fact is that, when one considers the sequence of principal minors of the Hankel matrix $\mathcal{M}_\theta$, all nonzero elements of this sequence can be characterized in terms of the simple continued-fraction expansion of $\theta$.

\begin{lemma}\label{lem:HM}
	Let $\theta\in\overline{\FF((X^{-1}))}$. Let $\mathcal{M}_\theta=(m_{i,j})_{i,j\geq 0}\in\FF^{\NN_0\times\NN_0}$ be the Hankel matrix associated with $\theta$. For $k\in\NN$, denote the submatrix corresponding to the $k$th principal minor of $\mathcal{M}_\theta$ by $\mathcal{M}^{(k)}_\theta=(m_{i,j})_{0\leq i,j< k}$. Then the following two assertions are equivalent. 
	\begin{enumerate}
		\item The matrix $\mathcal{M}^{(k)}_\theta\in\FF^{k\times k}$ is nonsingular.  
		\item The integer $k$ is the degree of the denominator of a best approximation to $\theta=\sum_{j=1}^\infty m_{0,j-1}X^{-j}$.  
	\end{enumerate}
\end{lemma}

\begin{proof}
	Cf. the proof of \cite[Lemma~2.3]{Hofer2026}.
\end{proof}

We provide details on Kronecker-type sequences (see Definition~\ref{def:KTseq}) and relate them to the general concept of digital sequences (see Remark~\ref{rem:KTseq}). Such sequences, constructed from counterexamples to \eqref{eq:Littlewood_series} or \eqref{eq:padicLittlewood_series}, are of particular interest because of their excellent uniform distribution properties. 

Kronecker-type sequences were introduced by Niederreiter (see, e.g., \cite{niesiam}) and studied in detail by Larcher and Niederreiter (see \cite{larnie1993,LarNie1993}) as analogues of the ordinary Kronecker sequence
$$((\{n\alpha_1\},\{n\alpha_2\},\ldots,\{n\alpha_s\}))_{n\geq 0},$$
where $(\alpha_1,\alpha_2,\ldots,\alpha_s)\in\RR^s$.

\begin{defi}[Kronecker-type sequence]\label{def:KTseq}
	Let $\FF_q$ be a finite field and let $\theta\in\overline{\Fq((X^{-1}))}$. For $n\in\NN_0$, the $n$th element $x_n$ of the Kronecker-type sequence determined by $\theta$ is obtained as follows:
	\begin{enumerate}
		\item Expand $n$ in base $q$, that is, write $n=\sum_{i=0}^\infty n_iq^i$ with $n_i\in\{0,1,\ldots,q-1\}$, and associate with $n$ the polynomial $n(X)\in\FF_q[X]$ defined by
		$$n(X)=\sum_{i=0}^\infty \eta(n_i)X^i,$$
		where $\eta:\{0,1,\ldots,q-1\}\to\FF_q$ is a fixed bijection satisfying $\eta(0)=0$. 
		
		\item Compute the product $n(X)\theta$ and consider the Laurent-series expansion of its fractional part,
		$$\langle n(X)\theta\rangle=\sum_{i=1}^\infty u_iX^{-i}.$$
		
		\item Finally, set
		$$x_n=\sum_{i=1}^\infty \eta^{-1}(u_i)q^{-i},$$
		which is an element of the interval $[0,1]$. We refer to the resulting sequence as
		$$\bigl([\langle n(X)\theta\rangle]_q\bigr)_{n\geq 0}.$$
	\end{enumerate}
More generally, given $(\theta_1,\ldots,\theta_s)\in
\overline{\Fq((X^{-1}))}^s$, the associated $s$-dimensional Kronecker-type sequence is
$$\bigl(([\langle n(X)\theta_1\rangle]_q,\ldots,[\langle n(X)\theta_s\rangle]_q)\bigr)_{n\geq 0},$$
which is a sequence in $[0,1]^s$.
\end{defi}

\begin{remark}\label{rem:KTseq}
	Note that the Kronecker-type sequence determined by $\theta$ can also be defined using the concept of an $\NN_0\times\NN_0$ generating matrix $\mathcal{M}_\theta$ over $\FF_q$ from Definition~\ref{def:Hankelmatrix}, together with the digital method introduced in \cite{nie1987}. The digital method works as follows. 
	
	We again fix a bijection $\eta:\{0,1,\ldots,q-1\}\to\FF_q$ satisfying $\eta(0)=0$, associate with each index $n\in\NN_0$, expanded in base $q$, the digit vector $\vec{n}=(\eta(n_0),\eta(n_1),\ldots)^T\in\FF_q^{\NN_0}$, and choose a generating matrix
	$$\mathcal{C}:=(c_{i,j})_{i\geq 0,j\geq 0}\in\FF_q^{\NN_0\times\NN_0}.$$
	We then compute
	$$\mathcal{C}\vec{n}=:\vec{y}_n=(y_{n,0},y_{n,1},\ldots)^T\in\FF_q^{\NN_0}.$$
	Finally, the $n$th point of the sequence $(x_n)_{n\geq 0}$ is defined by
	$$x_n:=\sum_{j=0}^\infty \eta^{-1}(y_{n,j})q^{-j-1}.$$
	
	If the generating matrix $\mathcal{C}$ is chosen to be the identity matrix, the resulting sequence is the well-known van der Corput sequence in base $q$ (see also Example~\ref{ex:vdc} in Section~\ref{sec:4}). 
	
	If, instead, the generating matrix $\mathcal{C}$ is chosen to be $\mathcal{M}_\theta$, then the digital method constructs the Kronecker-type sequence
	$$\bigl([\langle n(X)\theta\rangle]_q\bigr)_{n\geq 0},$$
	by Definition~\ref{def:Hankelmatrix} (see, e.g., \cite{LarNie1993,larnie1993}). 
	
	Although the Hankel matrix $\mathcal{M}_\theta$ can be defined over an arbitrary field $\FF$, the construction of a sequence in $[0,1]$ by this method relies on the finiteness of the field $\FF$.  
		
	More generally, an $s$-dimensional digital sequence is generated by
	$s$ matrices
	$\mathcal{C}^{(1)},\ldots,\mathcal{C}^{(s)}$ with
	$\mathcal{C}^{(i)}\in\FF_q^{\NN_0\times\NN_0}$ for $1\leq i\leq s$.
	Writing
	$$
	\mathcal{C}^{(i)}\vec{n}
	=(y_{n,0}^{(i)},y_{n,1}^{(i)},\ldots)^T,
	$$
	its $n$th point is
	$$
	\boldsymbol{x}_n
	=
	\left(
	\sum_{j=0}^{\infty}\eta^{-1}(y_{n,j}^{(1)})q^{-j-1},
	\ldots,
	\sum_{j=0}^{\infty}\eta^{-1}(y_{n,j}^{(s)})q^{-j-1}
	\right)\in[0,1]^s.
	$$
\end{remark}

In the theory of uniform distribution, the discrepancy of multidimensional Kronecker-type sequences is of particular interest. It is known that the $s$-dimensional Kronecker-type sequence
$$
\bigl([\langle n(X)\theta_1\rangle]_q,\ldots,[\langle n(X)\theta_s\rangle]_q\bigr)_{n\geq 0}
$$
has, in a certain sense, optimal distribution properties if $(\theta_1,\ldots,\theta_s)$ is a counterexample to the Littlewood conjecture \eqref{eq:padicLittlewood_series} over a finite field $\FF_q$ (see, e.g., \cite{LarNie1993,larnie1993}). Consequently, the series expansions of badly approximable elements in $\overline{\FF((X^{-1}))}$, where $\FF$ is a finite field, deserve particular attention.

Counterexamples to \eqref{eq:padicLittlewood_series} with $p(X)=X$ and $\FF=\FF_q$ can be used to construct two-dimensional sequences with excellent distribution properties (see \cite{levinLDhybrid}). The first component of such a sequence is the van der Corput sequence in base $q$, while the second is the Kronecker-type sequence associated with the counterexample $\theta$. Such sequences have been studied extensively in \cite{Hofer2026,Robertson}.\\

The remainder of this paper is organized as follows. Sections~\ref{sec:2} and~\ref{sec:3} generalize the results of \cite[Section~2]{Hofer2021} by introducing and analyzing the $\theta$-Ostrowski representation in $\FF[X]$ and an $\mathcal{L}_\theta\mathcal{I}_\theta\mathcal{U}_\theta$ decomposition of the Hankel matrix $\mathcal{M}_\theta$ defined in Definition~\ref{def:Hankelmatrix}. Section~\ref{sec:4} discusses the digital-construction viewpoint on potential counterexamples to the variants of Littlewood's conjecture in \eqref{eq:Littlewood_series} and \eqref{eq:padicLittlewood_series}. In particular, Sections~\ref{sec:42} and~\ref{sec:43} are devoted to these questions, whereas Section~\ref{sec:41} offers a new perspective on the celebrated Sobol' construction by providing a short and elegant proof of its excellent distribution properties.

\section{The $\theta$-Ostrowski expansion of a polynomial in $\FF[X]$}\label{sec:2}

An Ostrowski representation is a unique numeration system for the set of nonnegative integers $\NN_0$ based on the denominators $q_k$ of the convergents $p_k/q_k$ associated with the simple continued-fraction expansion of a fixed irrational number $\alpha=[a_0;a_1,a_2,\ldots]$. Every nonnegative integer $n$ can be written uniquely as
\[
n=\sum_{k=0}^{m}b_kq_k,
\]
where the coefficients $b_k$, which are computed by a greedy algorithm, satisfy the following conditions:

\begin{enumerate}
	\item $0\leq b_0<a_1$;
	\item $0\leq b_k\leq a_{k+1}$ for all $k\geq 1$;
	\item if $b_k$ attains its maximal value for some $k\geq 1$, that is, if $b_k=a_{k+1}$, then the preceding digit satisfies $b_{k-1}=0$.
\end{enumerate}

Every Ostrowski numeration system generalizes the $q$-adic numeration system with fixed base $q\geq 2$, in which every $n\in\NN_0$ has a unique representation
\[
n=\sum_{i=0}^{\infty}n_iq^i,
\qquad n_i\in\{0,1,\ldots,q-1\}.
\]
The $q$-adic representation is useful for analyzing the distribution properties of the van der Corput sequence in base $q$, whereas the Ostrowski representation based on $\alpha$ is used to analyze the discrepancy of Kronecker sequences $(\{n\alpha\})_{n\geq 0}$. Accordingly, it is natural to generalize the concept of an Ostrowski representation of the elements of $\NN_0$ based on $\alpha\in\RR$ to a corresponding Ostrowski representation of the elements of $\FF[X]$ based on a Laurent series $\theta\in\FF((X^{-1}))$.

One aim of this generalization is to transfer existing concepts to analogous settings. A further aim is to make the Ostrowski representation, which is useful for analyzing the sequence $(\{n\alpha\})_{n\geq 0}$, available as an additional perspective on the class of Kronecker-type sequences. The analysis of the distribution properties of Kronecker-type sequences is usually reduced to the rank properties of the associated Hankel matrices (see, e.g., \cite{LarNie1993,larnie1993}).

In what follows, we introduce the concept of an Ostrowski expansion of a polynomial in $\FF[X]$ (see Algorithm~\ref{algo:OstrRep}). Interestingly, the computation of an Ostrowski expansion in $\FF[X]$ can be described by a linear transformation on $\FF[X]$ (see Lemma~\ref{lem:LO_theta:expansion}). Finally, Lemma~\ref{lem:ValZec} explains the motivation for introducing Ostrowski expansions in connection with the variants of Littlewood's conjecture in \eqref{eq:Littlewood_series} and \eqref{eq:padicLittlewood_series}.

Let $\FF$ and $\theta\in\overline{\FF((X^{-1}))}$ be fixed. Let
\[
\left(\frac{P_k(X)}{Q_k(X)}\right)_{k\geq 0}
\]
be the sequence of convergents to $\theta$ associated with the simple continued-fraction expansion
\[
[0;A_1(X),A_2(X),A_3(X),\ldots]
\]
of $\theta$. Note that
\[
\deg (Q_k(X))=:d_k=\sum_{i=1}^{k}e_i,
\qquad e_i:=\deg A_i(X).
\]

We define the \emph{sequence of base polynomials} $\mathcal{B}_\theta$ in $\FF[X]$ associated with $\theta$ using the denominators of the convergents to $\theta$. This sequence of base polynomials will be used to compute the \emph{$\theta$-Ostrowski expansion} of a polynomial in $\FF[X]$.

\begin{defi}[Sequence of base polynomials $\mathcal{B}_\theta$ associated with $\theta$]\label{def:FibPol}
	Let
	\[
	\theta=[0;A_1(X),A_2(X),A_3(X),\ldots]\in\overline{\FF((X^{-1}))}
	\]
	be fixed, and for $k\in\NN_0$ let
	\[
	d_k:=\deg( Q_k(X))=\sum_{i=1}^{k}\deg( A_i(X))=\sum_{i=1}^{k}e_i
	\]
	denote the degree of the denominator $Q_k(X)$ of the $k$th convergent to $\theta$. The sequence of base polynomials
	\[
	\mathcal{B}_{\theta}:=(B_n(X))_{n\geq 0}
	\]
	in $\FF[X]$ is defined as follows. For each $n\in\NN_0$, let $k,v\in\NN_0$ be the unique integers satisfying
	\[
	d_k\leq n<d_{k+1},
	\qquad
	v:=n-d_k<e_{k+1}=\deg( A_{k+1}(X)).
	\]
	We then set
	\[
	B_n(X):=X^vQ_k(X).
	\]
	In particular,
	\[
	\deg( B_n(X))=v+\deg( Q_k(X))=v+d_k=n.
	\]
\end{defi}

Let $P(X)\in\FF[X]\setminus\{0\}$ have degree
\[
\deg (P(X))=r\in\NN_0.
\]
As usual, $P(X)$ can be written in the form
\[
P(X)=\sum_{i=0}^{r}p_iX^i,
\qquad p_i\in\FF\quad(0\leq i\leq r),
\qquad p_r\neq 0.
\]

We can now define the so-called \emph{$\theta$-Ostrowski representation} of $P(X)$ in $\FF[X]$ with respect to the sequence of base polynomials
\[
\mathcal{B}_{\theta}:=(B_n(X))_{n\geq 0}
\]
by means of the following greedy algorithm.

\begin{algo}[$\theta$-Ostrowski representation]\label{algo:OstrRep}
	Let $\theta\in\overline{\FF((X^{-1}))}$ and let $P(X)\in\FF[X]$. If $P(X)=0$, set
	\[
	(z_0,z_1,z_2,\ldots)=(0,0,0,\ldots)\in\FF^{\NN_0}.
	\]
	Otherwise, let $r:=\deg( P(X))\in\NN_0$. Proceed as follows:
	\begin{itemize}
		\item Set $P_r(X):=P(X)$. Choose $z_r\in\FF\setminus\{0\}$ such that
		\[
		P_{r-1}(X):=P_r(X)-z_rB_r(X)
		\]
		satisfies $\deg( P_{r-1}(X))<r$.
		
		\item For $i\in\{1,\ldots,r-1\}$, suppose that $P_i(X)$ has already been defined. If $\deg( P_i(X))<i$, set
		\[
		z_i:=0
		\qquad\text{and}\qquad
		P_{i-1}(X):=P_i(X).
		\]
		Otherwise, choose $z_i\in\FF\setminus\{0\}$ such that
		\[
		P_{i-1}(X):=P_i(X)-z_iB_i(X)
		\]
		satisfies $\deg (P_{i-1}(X))<i$.
		
		\item Finally, set
		\[
		z_0:=P_0(X).
		\]
		Since $\deg (P_0(X))<1$, we have $P_0(X)\in\FF$.
	\end{itemize}
	We call
	\[
	(z_0,z_1,\ldots,z_r,0,0,\ldots)\in\FF^{\NN_0}
	\]
	the \emph{digit vector} of the $\theta$-Ostrowski expansion of the polynomial $P(X)$, i.e. $$P(X)=\sum_{i=0}^{r}z_iB_i(X).$$
\end{algo}
It is immediate from Algorithm~\ref{algo:OstrRep} that the map
\[
(z_0,z_1,z_2,\ldots)\longmapsto \sum_{i=0}^{\infty}z_iB_i(X),
\]
where only finitely many $z_i$ are nonzero, defines a bijection between the finitely supported vectors in $\FF^{\NN_0}$ and the polynomial ring $\FF[X]$.

\begin{defi}[A nonsingular upper triangular (NUT) matrix $\mathcal{U}_\theta$]
	\label{def:U_theta_1}
	Let $\theta\in\overline{\FF((X^{-1}))}$. For each $j\in\NN_0$, let
	\[
	(z_{0,j},z_{1,j},z_{2,j},\ldots)\in\FF^{\NN_0}
	\]
	denote the digit vector of the $\theta$-Ostrowski expansion of $X^j$. We define
	\[
	\mathcal{U}_\theta:=(u_{i,j})_{i,j\geq 0}\in\FF^{\NN_0\times\NN_0}
	\]
	by
	\[
	u_{i,j}:=z_{i,j}.
	\]
	Thus, the $j$th column of $\mathcal{U}_\theta$ is the digit vector of the $\theta$-Ostrowski expansion of $X^j$.
\end{defi}

The following two lemmas are immediate consequences of basic linear algebra.

\begin{lemma}\label{lem:LO_theta:expansion}
	Let $\theta\in\overline{\FF((X^{-1}))}$, and let $\mathcal{U}_\theta$ be the matrix defined in Definition~\ref{def:U_theta_1}. Let
	\[
	P(X)\in\FF[X]\setminus\{0\}
	\]
	have degree $\deg (P(X))=r\in\NN_0$, and write
	\[
	P(X)=\sum_{i=0}^{r}p_iX^i,
	\qquad
	p_i\in\FF\quad(0\leq i\leq r),
	\qquad
	p_r\neq 0.
	\]
	If
	\[
	(z_0,z_1,\ldots,z_r,0,0,\ldots)\in\FF^{\NN_0}
	\]
	denotes the digit vector of the $\theta$-Ostrowski expansion of $P(X)$, then
	\[
	\mathcal{U}_\theta
	\begin{pmatrix}
		p_0\\
		p_1\\
		\vdots\\
		p_r\\
		0\\
		0\\
		\vdots
	\end{pmatrix}
	=
	\begin{pmatrix}
		z_0\\
		z_1\\
		\vdots\\
		z_r\\
		0\\
		0\\
		\vdots
	\end{pmatrix}.
	\]
\end{lemma}

\begin{lemma}\label{lem:inver_U_theta}
	Let $\theta\in\overline{\FF((X^{-1}))}$, and let $\mathcal{U}_\theta$ be the matrix defined in Definition~\ref{def:U_theta_1}. Then its inverse
	\[
	\mathcal{U}_\theta^{-1}=:\mathcal{V}_\theta
	=(v_{i,j})_{i,j\geq 0}
	\]
	is also a nonsingular upper triangular matrix. Moreover, $\mathcal{V}_\theta$ can be constructed columnwise from the sequence of base polynomials
	\[
	\mathcal{B}_\theta=(B_n(X))_{n\geq 0}.
	\]
	Indeed, writing
	\[
	B_j(X)=\sum_{i=0}^{j}b_{j,i}X^i,
	\]
	we have
	\[
	\mathcal{V}_\theta=(v_{i,j})_{i,j\geq 0}
	=(b_{j,i})_{i,j\geq 0}.
	\]
\end{lemma}

The next lemma establishes an important and useful property of the $\theta$-Ostrowski representation of a polynomial $P(X)\in\FF[X]$. Recall that
\[
\|P(X)\theta\|=e^{\nu(\langle P(X)\theta\rangle)}.
\]

\begin{lemma}\label{lem:ValZec}
	Let $\theta\in\overline{\FF((X^{-1}))}$, and let
	$P(X)\in\FF[X]\setminus\{0\}$ have associated digit vector
	\[
	(z_0,z_1,z_2,\ldots,z_r,0,0,\ldots)\in\FF^{\NN_0}
	\]
	in the $\theta$-Ostrowski representation. Then
	\[
	\nu\bigl(\langle P(X)\theta\rangle\bigr)\leq -d_k-1,
	\]
	where $k\in\NN_0$ is the smallest index such that
	\[
	(z_{d_k},z_{d_k+1},\ldots,z_{d_{k+1}-1})
	\neq (0,0,\ldots,0).
	\]
	Let
	\[
	u\in\{0,1,\ldots,\deg( A_{k+1}(X))-1\}
	\]
	be maximal such that
	\[
	z_{d_k+u}\neq 0.
	\]
	Then
	\[
	\nu\bigl(\langle P(X)\theta\rangle\bigr)
	=-d_{k+1}+u.
	\]
\end{lemma}

\begin{proof}
	For every $k\in\NN_0$, the approximation property of the convergents
	implies
	\[
	\nu\bigl(\langle Q_k(X)\theta\rangle\bigr)=-d_{k+1}.
	\]
	Since
	\[
	B_{d_k+v}(X)=X^vQ_k(X),
	\qquad
	0\leq v<e_{k+1}:=\deg( A_{k+1}(X)),
	\]
	we obtain
	\[
	\nu\bigl(\langle B_{d_k+v}(X)\theta\rangle\bigr)
	=v-d_{k+1}.
	\]
	
	By the definition of the $\theta$-Ostrowski representation,
	\[
	P(X)=\sum_{n=0}^{r}z_nB_n(X).
	\]
	Let $k$ be the smallest index for which the block
	\[
	(z_{d_k},z_{d_k+1},\ldots,z_{d_{k+1}-1})
	\]
	is nonzero, and let $u$ be maximal such that
	$z_{d_k+u}\neq 0$. The corresponding term has valuation
	\[
	\nu\bigl(z_{d_k+u}\langle B_{d_k+u}(X)\theta\rangle\bigr)
	=u-d_{k+1}.
	\]
	
	All nonzero terms from the same block have smaller valuation, since their
	indices are less than $d_k+u$. If $n$ belongs to a later block, say
	$n=d_\ell+v$ with $\ell\geq k+1$ and $0\leq v<e_{\ell+1}$, then
	\[
	\nu\bigl(\langle B_n(X)\theta\rangle\bigr)
	=v-d_{\ell+1}
	\leq -d_\ell-1
	\leq -d_{k+1}-1.
	\]
	Thus the term corresponding to $z_{d_k+u}B_{d_k+u}(X)$ has strictly
	largest valuation among all nonzero terms. Since $\nu$ is
	non-Archimedean, it follows that
	\[
	\nu\bigl(\langle P(X)\theta\rangle\bigr)=u-d_{k+1}.
	\]
	
	Finally, since $u\leq e_{k+1}-1$ and
	$d_{k+1}=d_k+e_{k+1}$, we obtain
	\[
	\nu\bigl(\langle P(X)\theta\rangle\bigr)
	=u-d_{k+1}
	\leq e_{k+1}-1-d_{k+1}
	=-d_k-1.
	\]
	This proves the lemma.
\end{proof}

\section{A block-$\mathcal{L}_\theta\mathcal{I}_\theta\mathcal{U}_\theta$ decompositions of a Hankel matrix $\mathcal{M}_\theta$  associated with $\theta\in\FF((X^{-1}))$}\label{sec:3}

In this section, we introduce a block-$\mathcal{L}_\theta\mathcal{I}_\theta
\mathcal{U}_\theta$ decomposition of the Hankel matrix $\mathcal{M}_\theta$
associated with $\theta\in\overline{\FF((X^{-1}))}$. In doing so, we extend
the results of \cite{Hofer2021}, which apply to a special subclass of
$\overline{\FF((X^{-1}))}$, to arbitrary $\theta\in\overline{\FF((X^{-1}))}$. The
results in \cite{Hofer2021} apply to those $\theta$ for which all partial
quotients in the continued-fraction expansion have degree one. This case
may be regarded as the extremal case within the set $\Bad$ defined in
Definition~\ref{def:set_Bad}, since it corresponds to the largest possible
value of
\[
\inf_{Q(X)\in\FF[X]\setminus\{0\}}
|Q(X)|\,\bigl\|Q(X)\theta(X)\bigr\|.
\]
However, restricting to this special class of $\theta$ may not be useful
when searching for counterexamples to \eqref{eq:Littlewood_series} and
\eqref{eq:padicLittlewood_series}; see, for example, the discussion in
\cite[Section~7.1]{Adiceametal}.

While the result in \cite{Hofer2021} yields a simple
$\mathcal{L}_\theta\mathcal{U}_\theta$ decomposition in the special case
described above, the Hankel matrix $\mathcal{M}_\theta$ associated with an
arbitrary $\theta$, as defined in Definition~\ref{def:Hankelmatrix},
requires an additional block-diagonal matrix $\mathcal{I}_\theta$ besides
the nonsingular upper- and lower-triangular matrices
$\mathcal{U}_\theta$ and $\mathcal{L}_\theta$. We denote this
block-diagonal matrix by $\mathcal{I}_\theta$ in order to avoid ambiguity
with Definition~\ref{def:FibPol}.

In accordance with \cite[Theorem~1, item~2]{Hofer2021}, we show that the
nonsingular upper-triangular matrix $\mathcal{U}_\theta$ coincides with the
matrix introduced in Definition~\ref{def:U_theta_1}, whose columns encode
the $\theta$-Ostrowski representations of the monomials $X^j$; see
Lemma~\ref{lem:LO_theta:expansion}. Furthermore, we introduce a simple
recursive algorithm for constructing $\mathcal{U}_\theta$ columnwise and
$\mathcal{L}_\theta$ rowwise. This algorithm is based on the thin matrices
$\mathcal{R}_\theta$ and $\mathcal{I}_\theta$, which are determined by the
simple continued-fraction expansion of $\theta$.

The motivation for introducing block-$\mathcal{L}_\theta\mathcal{I}_\theta
\mathcal{U}_\theta$ decompositions of Hankel matrices is the connection
between putative counterexamples $(\theta_1,\ldots,\theta_s)$ to
\eqref{eq:Littlewood_series} and a specific rank structure of the associated
Hankel matrices $\mathcal{M}_{\theta_1},\ldots,\mathcal{M}_{\theta_s}$; see
Remark~\ref{rem:CE_tvalue} in Section~\ref{sec:42}. Using scrambling
arguments from the digital method, this rank condition can be reduced to a
condition on the matrices
$\mathcal{U}_{\theta_1},\ldots,\mathcal{U}_{\theta_s}$; see, for example,
\cite{owen1997scrambled}, \cite{faure2002another}, and
Lemma~\ref{lem:scrambling} in Section~\ref{sec:41}. Thus, understanding
these matrices provides a new perspective on \eqref{eq:Littlewood_series}.

We first define the matrices $\mathcal{R}_{\theta}$ and $\mathcal{I}_\theta$ before explaining the construction of $\mathcal{U}_\theta$ and $\mathcal{L}_\theta$ in Definition~\ref{def:mat_UandL}.

\begin{defi}[Matrices $\mathcal{R}_{\theta}$ and $\mathcal{I}_{\theta}$ associated with $\theta$]\label{def:mat_RandI}
	
	We expand $\theta\in\overline{\FF((X^{-1}))}$ into its simple continued fraction
	$$\theta=[0;A_1(X),A_2(X),A_3(X),\ldots].$$ 
	For $k\in\NN$, we write as usual $$A_k(X)=a_{k,0}X^0+a_{k,1}X^1+\cdots+a_{k,\deg(A_k(X))}X^{\deg(A_k(X))},$$ and we set $e_k:=\deg(A_k(X))$ and $\lc(A_k(X)):=a_{k,e_k}$. 
	We define the matrix $\mathcal{R}_{\theta}\in\FF^{\NN_0\times\NN_0}$ associated with $\theta$ column by column, with indices starting from $0$.
	
	To determine the $n$th column for $n\in\NN_0$, we write $n=d_k+v$, where $k\in\NN_0$ is maximal such that $d_k\leq n$. This implies $0\leq v<e_{k+1}$. The $n$th column is then given by the $\theta$-Ostrowski representation of $X\cdot X^{v}Q_k(X)$. More precisely:
	\begin{itemize}
		\item[-] If $v\in\{0,\dots,e_{k+1}-2\}$, the $n$th column is given by 
		$$(\underbrace{0,0,\dots,0}_{n+1 \text{ entries}},1,0,0,\ldots)^T,$$
		\item[-] and if $v=e_{k+1}-1$, the $n$th column is determined by the $\theta$-Ostrowski representation of $X^{e_{k+1}}Q_k(X)$, which can be expressed as 
		\begin{align*}
			X^{e_{k+1}}Q_k(X)=&\lc(A_{k+1}(X))^{-1}\Big(Q_{k+1}(X)\\
			&-a_{k+1,e_{k+1}-1}X^{e_{k+1}-1}Q_{k}(X)\\
			&-a_{k+1,e_{k+1}-2}X^{e_{k+1}-2}Q_{k}(X)\\
			&-\cdots\\
			&-a_{k+1,0}X^{0}Q_{k}(X)\\
			&-Q_{k-1}(X)\Big).
		\end{align*}
		Consequently, the corresponding column of $\mathcal{R}_{\theta}$ is $\lc(A_{k+1}(X))^{-1}$ times
		$$(\underbrace{0,\dots,0}_{d_{k-1} \text{ entries}},-1,\underbrace{0,\dots,0}_{e_{k}-1 \text{ entries}},-a_{k+1,0},\dots,-a_{k+1,{e_{k+1}}-1},1,0,0,\ldots)^T.$$
		Note that if $k=0$, the initial sequence $0,\dots,0,-1,0,\dots,0$ does not appear. 
	\end{itemize}
	The matrix $\mathcal{I}_{\theta}$ is a block-diagonal matrix composed of a sequence of diagonal blocks, i.e., $\mathcal{I}_{\theta}=\operatorname{diag}(J_{e_{1}},J_{e_{2}},J_{e_{3}},\ldots)$, where $J_u$ (for $u\in\NN$) denotes the $u\times u$ anti-diagonal matrix:
	$$J_u=
	\begin{pmatrix}
		0      &  \dots  &	0   &    1 \\
		\vdots	&  \iddots &    \iddots  &   0\\
		0	&  1&   \iddots      &   \vdots\\
		1      & 0 & \dots   &  0  
	\end{pmatrix}
	\in \mathbb{F}^{u \times u}.$$ 
\end{defi}

\begin{defi}[Matrices $\mathcal{U}_{\theta}$ and $\mathcal{L}_{\theta}$ associated with $\theta$]\label{def:mat_UandL}
	
	Using the matrices $$\mathcal{R}_{\theta}, \quad \mathcal{I}_{\theta}, \quad \mbox{and} \quad \overline{\mathcal{R}}_\theta:=\mathcal{I}_{\theta}\mathcal{R}_{\theta}\mathcal{I}_{\theta},$$ we recursively construct two matrices: $\mathcal{U}_{\theta}$ columnwise and $\mathcal{L}_{\theta}$ rowwise.
	
	The matrix $\mathcal{U}_{\theta}\in\FF^{\NN_0\times\NN_0}$ is defined as follows. Its $0$th column is $\bsu_0=(1,0,0,0,\ldots)^T$, and, for $n\in\NN$, its $n$th column is given by
	$$\bsu_{n}=\mathcal{R}_{\theta}\bsu_{n-1}=\mathcal{R}_{\theta}^{n}\bsu_0.$$
	
	The matrix $\mathcal{L}_{\theta}$ is defined row by row. Its $0$th row is $\bsl_0=(1,0,0,0,\ldots)$, and, for $n\in\NN$, its $n$th row is given by
	$$\bsl_{n}=\bsl_{n-1}\overline{\mathcal{R}}_{\theta}=\bsl_0\overline{\mathcal{R}}_{\theta}^{n}=\bsl_0\mathcal{I}_{\theta}\mathcal{R}_{\theta}^{n}\mathcal{I}_{\theta}.$$
	
\end{defi}

We collect some nice as well as somehow astonishing properties of the matrices $\mathcal{R}_{\theta}$, $\mathcal{U}_{\theta}$, $\mathcal{L}_{\theta}$, and $\mathcal{I}_{\theta}$. We start with Lemma~\ref{lem:R_block} that immediately follows from Definition~\ref{def:mat_RandI}. 

\begin{lemma}\label{lem:R_block}
	The matrices $\mathcal{R}_{\theta}$ and
	$\overline{\mathcal{R}}_{\theta}$ are block tridiagonal. More
	precisely,
	\[
	\mathcal{R}_{\theta}=
	\begin{pmatrix}
		D_{1,1}&U_{1,2}&0&0&0&0&0&\cdots\\
		L_{2,1}&D_{2,2}&U_{2,3}&0&0&0&0&\cdots\\
		0&L_{3,2}&D_{3,3}&U_{3,4}&0&0&0&\cdots\\
		0&0&L_{4,3}&D_{4,4}&U_{4,5}&0&0&\cdots\\
		0&0&0&\ddots&\ddots&\ddots&0&\cdots
	\end{pmatrix}
	\]
	and
	\[
	\overline{\mathcal{R}}_{\theta}=
	\begin{pmatrix}
		\overline{D}_{1,1}&\overline{U}_{1,2}&0&0&0&0&0&\cdots\\
		\overline{L}_{2,1}&\overline{D}_{2,2}&\overline{U}_{2,3}&0&0&0&0&\cdots\\
		0&\overline{L}_{3,2}&\overline{D}_{3,3}&\overline{U}_{3,4}&0&0&0&\cdots\\
		0&0&\overline{L}_{4,3}&\overline{D}_{4,4}&\overline{U}_{4,5}&0&\cdots\\
		0&0&0&\ddots&\ddots&\ddots&0&\cdots
	\end{pmatrix},
	\]
	where
	\[
	D_{k,k},\overline{D}_{k,k}\in\FF^{e_k\times e_k},
	\qquad
	L_{k+1,k},\overline{L}_{k+1,k}
	\in\FF^{e_{k+1}\times e_k},
	\]
	and
	\[
	U_{k,k+1},\overline{U}_{k,k+1}
	\in\FF^{e_k\times e_{k+1}}.
	\]
	More specifically,
	\[
	U_{k,k+1}=
	\begin{pmatrix}
		0&\cdots&0&\blacksquare\\
		0&\cdots&0&0\\
		\vdots&&\vdots&\vdots\\
		0&\cdots&0&0
	\end{pmatrix}
	\in\FF^{e_k\times e_{k+1}},
	\qquad
	\overline{U}_{k,k+1}=
	\begin{pmatrix}
		0&0&\cdots&0\\
		\vdots&\vdots&&\vdots\\
		0&0&\cdots&0\\
		\blacksquare&0&\cdots&0
	\end{pmatrix}
	\in\FF^{e_k\times e_{k+1}},
	\]
	\[
	D_{k,k}=
	\begin{pmatrix}
		0&\cdots&0&\star\\
		\blacksquare&\ddots&\vdots&\vdots\\
		\vdots&\ddots&0&\star\\
		0&\cdots&\blacksquare&\star
	\end{pmatrix}
	\in\FF^{e_k\times e_k},
	\qquad
	\overline{D}_{k,k}=
	\begin{pmatrix}
		\star&\blacksquare&\cdots&0\\
		\star&0&\ddots&0\\
		\vdots&\ddots&\ddots&\blacksquare\\
		\star&0&\cdots&0
	\end{pmatrix}
	\in\FF^{e_k\times e_k},
	\]
	and
	\[
	L_{k+1,k}=
	\begin{pmatrix}
		0&\cdots&0&\blacksquare\\
		0&\cdots&0&0\\
		\vdots&&\vdots&\vdots\\
		0&\cdots&0&0
	\end{pmatrix}
	\in\FF^{e_{k+1}\times e_k},
	\qquad
	\overline{L}_{k+1,k}=
	\begin{pmatrix}
		0&0&\cdots&0\\
		\vdots&\vdots&&\vdots\\
		0&0&\cdots&0\\
		\blacksquare&0&\cdots&0
	\end{pmatrix}
	\in\FF^{e_{k+1}\times e_k}.
	\]
	Here, $\blacksquare$ denotes a nonzero element of $\FF$, whereas
	$\star$ denotes an element of $\FF$.
\end{lemma}

The objects introduced in this section, together with
Lemma~\ref{lem:R_block}, allow us to generalize the statement of
\cite[Theorem~1]{Hofer2021} as follows.

\begin{theorem}\label{prop:1}
	Let $\theta\in\overline{\FF((X^{-1}))}$. Then the following
	statements hold:
	\begin{enumerate}
		\item The matrix $\mathcal{L}_{\theta}$ defined in
		Definition~\ref{def:mat_UandL} is a nonsingular lower triangular
		matrix, and $\mathcal{U}_{\theta}$ is a nonsingular upper triangular
		matrix. 
		
		\item The $j$th column of $\mathcal{U}_{\theta}$, as defined in
		Definition~\ref{def:mat_UandL}, consists of the coefficients in the
		$\theta$-Ostrowski representation of $X^j$. More precisely, if
		\[
		X^j=\sum_{n=0}^{j}z_nB_n(X),
		\] with $
		(B_n(X))_{n\geq 0}=\mathcal{B}_{\theta}$,
		given in Definition~\ref{def:FibPol}, then
		\[
		\bsu_j=(z_0,z_1,z_2,\ldots,z_j,0,\ldots)^{\mathsf T}.
		\]
		Consequently, $\mathcal{U}_{\theta}$ coincides with the matrix
		associated with $\theta$ in Definition~\ref{def:U_theta_1}.
		
		\item The product
		\[
		\mathcal{L}_{\theta}\mathcal{I}_{\theta}\mathcal{U}_{\theta}
		\]
		is a Hankel matrix.
		
		\item The product
		\[
		\lc(A_1(X))^{-1}
		\mathcal{L}_{\theta}\mathcal{I}_{\theta}\mathcal{U}_{\theta}
		\]
		is the Hankel matrix $\mathcal{M}_{\theta}$ associated with
		$\theta$.
	\end{enumerate}
\end{theorem}

\begin{proof}
	For the first item, note that
	Lemma \ref{lem:R_block} shows that the matrix $\mathcal{R}_{\theta}$ is an almost upper triangular matrix with nonzero entries in the lower diagonal part only in the first subdiagonal which contains only $\blacksquare$s and the matrix $\overline{\mathcal{R}}_{\theta}$ is an almost lower triangular matrix with nonzero entries in the upper diagonal part only in the first supdiagonal which contains only $\blacksquare$s. Hence the first item is an easy consequence using the definition of $\mathcal{U}_\phi$ and $\mathcal{L}_\phi$.\\
	
	We next prove the Hankel property in item 3. Write
	\[
	\mathcal{L}_{\theta}\mathcal{I}_{\theta}\mathcal{U}_{\theta}
	=:(k_{i,j})_{i,j\geq 0}.
	\]
	By the definitions of $\mathcal{L}_{\theta}$ and
	$\mathcal{U}_{\theta}$, and using the fact that
	$\mathcal{I}_{\theta}$ is self-inverse, i.e. $\mathcal I_\theta\mathcal I_\theta=\mathcal I$, we obtain
	\begin{align*}
		k_{i,j}
		&=\bsl_i\mathcal{I}_{\theta}\bsu_j\\
		&=\bsl_0\mathcal{I}_{\theta}
		\mathcal{R}_{\theta}^{\,i}
		\mathcal{I}_{\theta}
		\mathcal{I}_{\theta}
		\mathcal{R}_{\theta}^{\,j}\bsu_0\\
		&=\bsl_0\mathcal{I}_{\theta}
		\mathcal{R}_{\theta}^{\,i-s}
		\mathcal{I}_{\theta}
		\mathcal{I}_{\theta}
		\mathcal{R}_{\theta}^{\,j+s}\bsu_0\\
		&=\bsl_{i-s}\mathcal{I}_{\theta}\bsu_{j+s}\\
		&=k_{i-s,j+s},
	\end{align*}
	for all admissible $i,j$ and $s$. Thus, $k_{i,j}$ depends only on
	$i+j$, and consequently
	$\mathcal{L}_{\theta}\mathcal{I}_{\theta}\mathcal{U}_{\theta}$
	is a Hankel matrix.
	
	We now prove item 2 by induction on $j$. For $j=0$, the assertion
	is immediate. Suppose that, for some $j\in\NN_0$, the $j$th column
	$\bsu_j$ of $\mathcal{U}_{\theta}$ has the form
	\[
	\bsu_j
	=
	\bigl(
	z^{(j)}_0,z^{(j)}_1,z^{(j)}_2,\ldots,z^{(j)}_j,0,\ldots
	\bigr)^{\mathsf T},
	\]
	where
	\[
	X^j=\sum_{i=0}^{j}z^{(j)}_iB_i(X).
	\]
	It remains to show that the coefficients of the
	$\theta$-Ostrowski representation of $X^{j+1}$ are obtained from
	those of $X^j$ by multiplication with $\mathcal{R}_{\theta}$; that
	is,
	\begin{equation}\label{eq:toshow}
		\bigl(z^{(j+1)}_0,z^{(j+1)}_1,\ldots\bigr)^{\mathsf T}
		=
		\mathcal{R}_{\theta}
		\bigl(z^{(j)}_0,z^{(j)}_1,\ldots\bigr)^{\mathsf T}.
	\end{equation}
	
	Using the definition of the $\theta$-Ostrowski representation and
	the decomposition of the basis polynomials, we have
	\begin{align*}
		X^{j+1}
		&=XX^j\\
		&=X\sum_{i=0}^{\infty}z^{(j)}_iB_i(X)\\
		&=X\sum_{k=0}^{\infty}
		\sum_{v=0}^{e_{k+1}-1}
		z^{(j)}_{d_k+v}X^vQ_k(X)\\
		&=\sum_{k=0}^{\infty}
		\sum_{v=0}^{e_{k+1}-1}
		z^{(j)}_{d_k+v}X^{v+1}Q_k(X)\\
		&=\sum_{k=0}^{\infty}
		\sum_{v=0}^{e_{k+1}-2}
		z^{(j)}_{d_k+v}X^{v+1}Q_k(X)
		+
		\sum_{k=0}^{\infty}
		z^{(j)}_{d_k+e_{k+1}-1}X^{e_{k+1}}Q_k(X).
	\end{align*}
	The last expression is precisely the recursion defining
	$\mathcal{R}_{\theta}$. Hence \eqref{eq:toshow} follows, completing
	the induction. Therefore, the $j$th column of $\mathcal{U}_{\theta}$
	contains the coefficients of the $\theta$-Ostrowski representation
	of $X^j$. This proves the second.
	
	It remains to prove item 4. Since both
	$\mathcal{M}_{\theta}$ and
	$\mathcal{L}_{\theta}\mathcal{I}_{\theta}\mathcal{U}_{\theta}$
	are Hankel matrices, it suffices to compare one suitable row of each
	matrix. More precisely, it is enough to show that the zeroth row of
	$\mathcal{M}_{\theta}$ agrees with
	$\lc(A_1(X))^{-1}$ times the row of
	$\mathcal{U}_{\theta}$ which is indexed by $(e_1-1)$.
	
	Write
	\[
	\mathcal{U}_{\theta}=(u_{i,j})_{i,j\geq 0}
	\]
	and write the Laurent expansion of $\theta$ as
	\[
	\theta=\sum_{i=1}^{\infty}c_iX^{-i}.
	\]
	We first determine the first $e_1$ coefficients. The first convergent
	to $\theta$ is $1/A_1(X)$. Consequently,
	\begin{equation}\label{eq:item4_1}
		(c_1,c_2,\ldots,c_{e_1})
		=
		(0,0,\ldots,0,\lc(A_1(X))^{-1}).
	\end{equation}
	
	For every $m\in\NN_0$, we have
	\[
	\nu\bigl(\langle X^m\theta-c_{m+1}X^{-1}\rangle\bigr)<-1.
	\]
	Let
	\[
	P(X)=X^m-z^{(m)}_{e_1-1}X^{e_1-1},
	\]
	where
	\[
	X^m=\sum_{i=0}^{m}z^{(m)}_iB_i(X)
	\]
	is the $\theta$-Ostrowski representation of $X^m$. Applying
	Lemma~\ref{lem:ValZec} to $P(X)$ gives
	\[
	\nu\bigl(
	\langle (X^m
	-z^{(m)}_{e_1-1}X^{e_1-1})\theta\rangle
	\bigr)=\nu\bigl(
	\langle X^m\theta
	-z^{(m)}_{e_1-1}X^{e_1-1}\theta\rangle
	\bigr)<-1.
	\]
	Together with \eqref{eq:item4_1}, the inequalities
	\[
	\nu\bigl(\langle X^m\theta-c_{m+1}X^{-1}\rangle\bigr)<-1
	\]
	and
	\[
	\nu\bigl(\langle X^{e_1-1}\theta-c_{e_1}X^{-1}\rangle\bigr)<-1,
	\]
	as well as part~(2), this implies
	\[
	c_{m+1}
	=
	z^{(m)}_{e_1-1}c_{e_1}
	=
	u_{e_1-1,m}c_{e_1}
	=
	\lc(A_1(X))^{-1}u_{e_1-1,m}.
	\]
	Thus, the zeroth row of $\mathcal{M}_{\theta}$ coincides with
	$\lc(A_1(X))^{-1}$ times the row of
	$\mathcal{U}_{\theta}$ indexed by $(e_1-1)$. Since both matrices are Hankel, we conclude
	that
	\[
	\mathcal{M}_{\theta}
	=
	\lc(A_1(X))^{-1}
	\mathcal{L}_{\theta}\mathcal{I}_{\theta}\mathcal{U}_{\theta}.
	\]
\end{proof}

\section{The digital construction point of view} \label{sec:4}

The original motivation was to translate the search for counterexamples to variants of Littlewood's conjecture, namely \eqref{eq:Littlewood_series} and \eqref{eq:padicLittlewood_series}, into the search for specific Kronecker-type sequences. The hope was that viewing these sequences from the perspective of digital sequences would lead to new approaches to, or perspectives on, such counterexamples. Interestingly, this study of Kronecker-type sequences, motivated by Littlewood's conjecture, now also offers a new perspective on the famous Sobol' construction \cite{sobol}. Sobol' sequences constitute one of the most popular classes of so-called low-discrepancy sequences, which are used in quasi-Monte Carlo methods. 

Our new perspective on the Sobol' construction makes it particularly easy to generalize the original construction, which uses $s$ distinct primitive polynomials over the binary field $\FF_2$, to one based on any $s$ monic, pairwise coprime polynomials of degree at least $1$ over an arbitrary finite field $\FF_q$. This generalization was previously obtained by Faure and Lemieux \cite{FaureLemieux2016} for irreducible polynomials. Their proof relates the resulting so-called irreducible Sobol' sequences to generalized Niederreiter sequences \cite{nie1987}; see also \cite{Tez2}. Within our perspective on the Sobol' construction, the proof of the good distribution properties of the Sobol' sequences considered in this paper can be presented very briefly and elegantly using identities involving their generating matrices.

\subsection{A construction in the spirit of Sobol' and a new and easy proof of its the excellent distribution properties}\label{sec:41}

First, we introduce the concept of digital $(u,\bse,s)$-sequences in the sense of Tezuka \cite{Tez2}, which generalizes digital $(t,s)$-sequences in the sense of Niederreiter (see, e.g., \cite{niesiam}).

\begin{defi}\label{def:dig_ues_seq}
	Let $\mathcal{C}^{(1)},\ldots,\mathcal{C}^{(s)}\in\FF_q^{\NN_0\times \NN_0}$, and let $(\bsx_n)_{n\geq 0}$ be the sequence in $[0,1]^s$ which is constructed from $\mathcal{C}^{(1)},\ldots,\mathcal{C}^{(s)}$ as described in Remark \ref{rem:KTseq}. The sequence $(\bsx_n)_{n\geq 0}$ is called a digital $(u,\bse,s)$-sequence over $\FF_q$ with quality parameter $u\geq 0$ and $\bse=(e_1,\ldots,e_s)\in\NN^s$ if, for every choice of nonnegative integers $l_1,\ldots,l_s$, the $(m-u)\times m$ matrix over $\FF_q$ formed by the row vectors
	$$(c^{(i)}_{j,0},c^{(i)}_{j,1},\ldots,c^{(i)}_{j,m-1} )$$
	with $j\in\{0,\ldots,e_il_i-1\}$ and $i\in\{1,\ldots,s\}$ has full row rank $m-u$, where
	$$m:=l_1e_1+\cdots+l_se_s+u.$$
\end{defi}

\begin{remark}\label{rem:tez_u_t}
	Note that a $(u,(e_1,\ldots,e_s),s)$-sequence is also a $(t,(1,\ldots,1),s)$-sequence, where
	$$t=u+\sum_{i=1}^s e_i-1.$$
	The latter is commonly referred to as a $(t,s)$-sequence in the sense of Niederreiter. For further information on digital sequences, see, for example, \cite{Tez2} and \cite{niesiam}.
\end{remark}

The following lemma presents several possible modifications of the generating matrices that preserve the $(u,\bse,s)$ property of the resulting digital sequence. 
\begin{lemma}\label{lem:scrambling}
	Let $\mathcal{C}^{(1)},\ldots,\mathcal{C}^{(s)}\in\FF_q^{\NN_0\times \NN_0}$ be qualified to construct a digital $(u,\bse,s)$-sequence over $\FF_q$. Let $\mathcal{U}\in\FF_q^{\NN_0\times \NN_0}$ be a NUT matrix and let $\mathcal{L}^{(1)},\ldots,\mathcal{L}^{(s)}\in\FF_q^{\NN_0\times \NN_0}$ be NLT (non-singular lower triangular) matrices. 
	\begin{itemize}
		\item Then  $\mathcal{L}^{(1)}\mathcal{C}^{(1)},\ldots,\mathcal{L}^{(s)}\mathcal{C}^{(s)}\in\FF_q^{\NN_0\times \NN_0}$ are qualified to construct a digital $(u,\bse,s)$-sequence over $\FF_q$. (Owen scrambling)
		\item Then  $\mathcal{C}^{(1)}\mathcal{U},\ldots,\mathcal{C}^{(s)}\mathcal{U}\in\FF_q^{\NN_0\times \NN_0}$ are qualified to construct a digital $(u,\bse,s)$-sequence over $\FF_q$. (Faure-Tezuka scrambling)
	\end{itemize}
	\end{lemma}
	\begin{proof}
		Both follow from Definition~\ref{def:dig_ues_seq} together with basic linear algebra. The first item was introduced already in \cite{owen1997scrambled}. The second in \cite{faure2002another}. 
	\end{proof}

Definition~\ref{def:sobol_construction} provides a columnwise method for constructing the generating matrices $\mathcal{C}^{(1)},\ldots,\mathcal{C}^{(s)}\in\FF_q^{\NN_0\times \NN_0}$. This construction is analogous to the columnwise construction of $\mathcal{U}_{\theta}$ in Definition~\ref{def:mat_UandL} and can be viewed as a generalization of both the Sobol' construction \cite{sobol} and the subsequently introduced irreducible Sobol' sequences of Faure and Lemieux \cite{FaureLemieux2016}.

\begin{defi}\label{def:sobol_construction}
	Let $q=p^r$, where $p\in\PP$, $r\in\NN$, and $s\in\NN$. Choose $s$ pairwise coprime monic polynomials
	\[
	b_1(X),\ldots,b_s(X)\in\FF_q[X]
	\]
	of degrees $e_1,\ldots,e_s\geq 1$.
	
	For each $i\in\{1,\ldots,s\}$, construct a matrix
	$\mathcal{C}^{(i)}\in\FF_q^{\NN_0\times\NN_0}$ using the base polynomial
	\[
	b_i(X)=X^{e_i}+a^{(i)}_{e_i-1}X^{e_i-1}+\cdots+a^{(i)}_1X+a^{(i)}_0,
	\]
	where $a^{(i)}_{e_i-1},\ldots,a^{(i)}_1,a^{(i)}_0\in\FF_q$, by recursively defining its columns. We begin with the $e_i$ initial columns
	\[
	\begin{pmatrix}|&&|\\
		c^{(i)}_0&\cdots&c^{(i)}_{e_i-1}\\
		|&&|
	\end{pmatrix}
	=
	\begin{pmatrix}
		B_i\\
		0\\
		0\\
		\vdots
	\end{pmatrix}
	\in\FF_q^{\NN_0\times e_i},
	\]
	where $B_i$ is an $e_i\times e_i$ NUT matrix over $\FF_q$. All subsequent columns are defined recursively by
	\begin{equation}\label{equ:sobol_construction}
		c^{(i)}_{e_i+r}
		+a^{(i)}_{e_i-1}c^{(i)}_{e_i-1+r}
		+\cdots
		+a^{(i)}_1c^{(i)}_{1+r}
		+a^{(i)}_0c^{(i)}_{r}
		=
		\begin{pmatrix}
			0\\
			\vdots\\
			0\\
			c^{(i)}_{r}
		\end{pmatrix},
	\end{equation}
	where the column on the right-hand side begins with $e_i$ zero entries.
	
	We call the sequence $(\bsx_n)_{n\geq 0}$ constructed from the matrices
	$\mathcal{C}^{(1)},\ldots,\mathcal{C}^{(s)}$ via the algorithm described in Remark~\ref{rem:KTseq} the \emph{Sobol'-type sequence} associated with the base polynomials
	$b_1(X),\ldots,b_s(X)\in\FF_q[X]$.
\end{defi}

First, we observe that the matrices $\mathcal{C}^{(1)},\ldots,\mathcal{C}^{(s)}$ are NUT matrices; this follows immediately from Definition~\ref{def:sobol_construction}.

\begin{lemma}\label{lem:inversematrices}
	The matrices $\mathcal{C}^{(1)},\ldots,\mathcal{C}^{(s)}$ constructed in Definition~\ref{def:sobol_construction} are NUT matrices in $\FF_q^{\NN_0\times\NN_0}$. Consequently, each matrix has a unique NUT inverse, which we denote by
	\[
	(\mathcal{C}^{(1)})^{-1}=:\mathcal{C}^{(-1)},\ldots,
	(\mathcal{C}^{(s)})^{-1}=:\mathcal{C}^{(-s)}.
	\]
\end{lemma}

Theorem~\ref{thm:sobol} can be proved elegantly by exploiting identities involving the generating matrices and their inverses.

\begin{theorem}\label{thm:sobol}
	The matrices $\mathcal{C}^{(1)},\ldots,\mathcal{C}^{(s)}$ constructed in Definition~\ref{def:sobol_construction} generate a digital $(0,\bse,s)$-sequence over $\FF_q$, where
	\[
	\bse=(e_1,\ldots,e_s).
	\]
\end{theorem}

\begin{proof}
	Let
	\[
	\mathcal{X}:=
	\begin{pmatrix}
		\bszero\\
		\mathcal{I}
	\end{pmatrix},
	\]
	where $\mathcal{I}$ denotes the infinite identity matrix over $\FF_q$. By the recursion in \eqref{equ:sobol_construction}, the matrix $\mathcal{C}^{(i)}$ satisfies
	\[
	\mathcal{C}^{(i)}b_i(\mathcal{X})
	=
	\mathcal{X}^{e_i}\mathcal{C}^{(i)}.
	\]
	Multiplying this identity from the left by $(\mathcal{C}^{(i)})^{-1}=\mathcal{C}^{(-i)}$ and from the right by $\mathcal{C}^{(-i)}$ gives
	\begin{equation}\label{eq:sobol_construction}
		b_i(\mathcal{X})\mathcal{C}^{(-i)}
		=
		\mathcal{C}^{(-i)}\mathcal{X}^{e_i}.
	\end{equation}
	In particular, if
	\[
	\mathcal{C}^{(-i)}
	=
	\begin{pmatrix}
		|&|&&\\
		c^{(-i)}_0&c^{(-i)}_1&\cdots&\\
		|&|&&
	\end{pmatrix},
	\]
	then \eqref{eq:sobol_construction} implies
	\[
	b_i(\mathcal{X})c^{(-i)}_j=c^{(-i)}_{j+e_i},
	\qquad j\in\NN_0.
	\]
	
	For a polynomial
	\[
	p(X)=\sum_{j=0}^r a_jX^j\in\FF_q[X],
	\]
	we write
	\[
	\vec{p}:=(a_0,a_1,\ldots,a_r,0,0,\ldots)^{\mathsf T}\in\FF_q^{\NN_0}.
	\]
	Then
	\begin{equation}\label{eq:identity_vec_Xi}
		\mathcal{I}\vec{p}
		=
		p(\mathcal{X})
		\begin{pmatrix}
			1\\0\\0\\\vdots
		\end{pmatrix}.
	\end{equation}
	
	We now verify the linear independence criterion from Definition~\ref{def:dig_ues_seq}. Let $l_1,\ldots,l_s\in\NN_0$ and put
	\[
	m:=l_1e_1+\cdots+l_se_s.
	\]
	We have to show that the $m$ row vectors
	\[
	\bigl(c^{(i)}_{j,0},c^{(i)}_{j,1},\ldots,c^{(i)}_{j,m-1}\bigr),
	\qquad
	i\in\{1,\ldots,s\},\quad
	j\in\{0,\ldots,e_il_i-1\},
	\]
	are linearly independent over $\FF_q$.
	
	By Definition~\ref{def:dig_ues_seq}, it suffices to show that the
common kernel of the corresponding $m\times m$ matrix is trivial.
Let $\vec n$ belong to this kernel. Thus, $\vec n$ is the coefficient
vector of a polynomial $n(X)\in\FF_q[X]$ with $\deg(n)<m$, and the
first $e_il_i$ entries of $\mathcal{C}^{(i)}\vec n$ vanish for every
$i\in\{1,\ldots,s\}$. Therefore, there exists
$v^{(i)}(X)\in\FF_q[X]$ such that
\[
\mathcal{C}^{(i)}\vec n=\mathcal{X}^{e_il_i}\vec v^{(i)}.
\]
	Since $\mathcal{C}^{(-i)}\vec{v}^{(i)}$ is again the coefficient vector of a polynomial, say $w^{(i)}(X)$, it follows from \eqref{eq:sobol_construction} and \eqref{eq:identity_vec_Xi} that
	\[
	n(\mathcal{X})
	=
	b_i(\mathcal{X})^{l_i}w^{(i)}(\mathcal{X}).
	\]
	Consequently,
	\[
	b_i(X)^{l_i}\mid n(X)
	\qquad\text{for every }i\in\{1,\ldots,s\}.
	\]
	Because the polynomials $b_1(X),\ldots,b_s(X)$ are pairwise coprime, their powers are pairwise coprime, and hence
	\[
	\prod_{i=1}^s b_i(X)^{l_i}\mid n(X).
	\]
	The degree of the polynomial on the left-hand side is
	\[
	\sum_{i=1}^s l_i\deg(b_i(X))
	=
	\sum_{i=1}^s l_ie_i
	=
	m.
	\]
	Since $\deg(n(X))<m$, this is possible only if $n(X)=0$. Therefore, the stated row vectors are linearly independent.
	
	By Definition~\ref{def:dig_ues_seq}, the matrices
	$\mathcal{C}^{(1)},\ldots,\mathcal{C}^{(s)}$ generate a digital
	$(0,\bse,s)$-sequence over $\FF_q$, where
	\[
	\bse=(e_1,\ldots,e_s).
	\]
\end{proof}

\begin{example}\label{ex:vdc}
	\begin{enumerate}
		\item If $s=1$, $b_1(X)=X$, and $B_1=(1)$, then Definition~\ref{def:sobol_construction} yields
		\[
		\mathcal{C}^{(1)}=\mathcal{I}
		\]
		in $\FF_q^{\NN_0\times\NN_0}$. Consequently, the algorithm from Remark~\ref{rem:KTseq} generates the van der Corput sequence in base $q$,
		\[
		\bigl(\varphi_q(n)\bigr)_{n\geq 0},
		\]
		where $\varphi_q:\NN_0\to[0,1)$ denotes the $q$-adic radical inverse function. More precisely, if
		\[
		n=n_0+n_1q+n_2q^2+\cdots
		\]
		is the unique base-$q$ representation of $n$, then
		\[
		\varphi_q(n)
		=
		\frac{n_0}{q}
		+\frac{n_1}{q^2}
		+\frac{n_2}{q^3}
		+\cdots.
		\]
		
		\item If $s=1$, $b_1(X)=X+a$ with $a\in\FF_q^*$, and $B_1=(1)$, then Definition~\ref{def:sobol_construction} yields
		\[
		\mathcal{C}^{(1)}
		=
		\left(
		\binom{j}{i}(-a)^{j-i}
		\right)_{i,j\geq 0}.
		\]
		Moreover, the $j$th column of $\mathcal{C}^{(-1)}$ is given by the coefficient vector of
		\[
		b_1(X)^j=(X+a)^j
		=
		\sum_{i=0}^j \binom{j}{i}a^{j-i}X^i,
		\]
		that is,
		\[
		\mathcal{C}^{(-1)}
		=
		\left(
		\binom{j}{i}a^{j-i}
		\right)_{i,j\geq 0}.
		\]
		Furthermore, the proof of Theorem~\ref{thm:sobol} shows that, for every $l\in\NN_0$ and every polynomial $n(X)\in\FF_q[X]$,
		\[
		\mathcal{C}^{(1)}\vec n
		=
		\mathcal{X}^{l}\vec v
		\quad\text{for some }\vec v
		\]
		if and only if
		\[
		(X+a)^l\mid n(X).
		\]
	\end{enumerate}
\end{example}

\subsection{Kronecker-type and related sequences and (possible) counterexamples of variants of the Littlewood conjecture}\label{sec:42}
The different variants of Littlewood's conjecture over $\FF_q((X^{-1}))$ can be linked to Kronecker-type sequences and related sequences (see, e.g., \cite{Hofer2026}). These connections motivate the study of Kronecker-type sequences in this paper, with particular emphasis on their generating Hankel matrices. The remainder of the paper presents several observations, discussions, and open problems.

\begin{example}\label{ex:1D_Kron_type} Let $\theta \in\overline{\FF_q((X^{-1}))}$ be badly approximable. We have the following three possibilies to construct a digital sequence. 
	\begin{enumerate}
		\item We use $\mathcal{C}=\mathcal{M}_\theta$ due to Definition~\ref{def:Hankelmatrix}. Then from 
		the theory of continued fractions we derive 
		\begin{equation}\label{eq:badlyapprox_series}
			\inf_{n(X)\in\FF_q[X]\setminus\{0\}}|n(X)|\cdot \|n(X)\theta\|= e^{-\sup_{k\geq 1}(\deg(A_{k+1}))}.
		\end{equation} We set $$t=\sup_{k\geq 1}(\deg(A_{k+1}))-1.$$ Lemma~\ref{lem:HM} shows then that the digital sequences described in Remark \ref{rem:KTseq} is a $(t,1)$-sequence over $\FF_q$. 
		\item We use $\mathcal{C}=\mathcal{I}_\theta\mathcal{U}_\theta$ due to Definitions~\ref{def:mat_RandI} and \ref{def:mat_UandL}. We set $$t=\sup_{k\geq 1}(\deg(A_{k+1}))-1.$$ Then from item~1, Lemma~\ref{lem:scrambling}, and Theorem~\ref{prop:1} we know that the digital sequences described in Remark \ref{rem:KTseq} is a $(t,1)$-sequence. 
			\item We use $\mathcal{C}=\mathcal{U}_\theta$. Then immediately from the non-singular upper triangular property we know that the  digital sequences described in Remark \ref{rem:KTseq} is a $(0,1)$-sequence. 
	\end{enumerate}

\end{example}

	All three types of sequences in Example~\ref{ex:1D_Kron_type} yield one-dimensional low-discrepancy sequences. The analogue of equation \eqref{eq:badlyapprox_series} over $\RR$ is a badly approximable $\alpha\in\RR$, that is,
$$\inf_{h\in\ZZ\setminus\{0\}}|h|\cdot \|h\alpha\|= c_{\alpha}>0.$$
The analogue of Item~1 in Example \ref{ex:1D_Kron_type} is the $(n\alpha)$-sequence, which has been studied extensively (see \cite{DT} and the references therein). An analogue in $\RR$ of Items~2 and~3 in Example \ref{ex:1D_Kron_type} is known only in the case where $\alpha=\frac{\sqrt{5}-1}{2}=:\varphi$, the golden ratio. This analogue (see \cite{Drmota} and the references therein) is based on the so-called Zeckendorf expansions of $n$, together with a suitable radical inverse function comparable to the $q$-adic radical inverse function (see Example~\ref{ex:vdc}). In \cite{Drmota}, the distribution of this \emph{Zeckendorf--Van der Corput sequence} is related to that of the $(n\varphi)$-sequence. It remains an open problem to define appropriate Ostrowski--van der Corput sequences in $\RR$ and to investigate whether analogous relations hold between the distributions of these sequences and the corresponding $(n\alpha)$-sequences.

\begin{remark}\label{rem:CE_tvalue} Using $\theta_1,\ldots,\theta_s \in\overline{\FF_q((X^{-1}))}$ all badly approximable. We have again three possibilies to construct a digital sequence. 
	\begin{enumerate}
		\item We use $\mathcal{C}^{(1)}=\mathcal{M}_{\theta_1},\ldots,\mathcal{C}^{(s)}=\mathcal{M}_{\theta_s}$. Then from \cite[Theorem~2]{larnie1993} we derive that the Kronecker-type sequence, forms a $(t,s)$-sequence with $t\geq 0$
		 if and only if 
		\begin{equation}\label{eq:Littlewood_series_2}\inf_{(Q_1,\ldots,Q_s)\in\FF[X]^s\setminus\{\boldsymbol{0}\}} \max\{\frac{1}{e},|Q_1|\}\cdots \max\{\frac{1}{e},|Q_s|\} \cdot\|Q_1\theta_1+\cdots+Q_s\theta_s\|\geq e^{-t-s},
		\end{equation}
		with $t\geq 0$. 
		
		The latter holds only for a $t\in\NN_0$ if and only if $(\theta_1,\ldots,\theta_s)\in\overline{\FF_q((X^{-1}))}^s$ is a counterexample to \eqref{eq:Littlewood_series}. 
	
	\item We use $\mathcal{C}^{(1)}=\mathcal{I}_{\theta_1}\mathcal{U}_{\theta_1},\ldots,\mathcal{C}^{(s)}=\mathcal{I}_{\theta_s}\mathcal{U}_{\theta_s}$, where $(\theta_1,\ldots,\theta_s)\in\overline{\FF_q((X^{-1}))}^s$ is a counterexample to \eqref{eq:Littlewood_series}. 
The $t$-value is the same as the one above, since multiplying $\mathcal{C}^{(i)}$ on the left by $\lc(A_1^{(i)}(X))^{-1}\mathcal{L}_{\theta_i}$, for $i=1,\ldots,s$, can be viewed as applying the Owen scrambling described in Lemma~\ref{lem:scrambling} to the sequence. Note that Theorem~\ref{prop:1} implies that the scrambled sequence coincides with the sequence in the preceding item.

		\item  We use $\mathcal{C}^{(1)}=\mathcal{U}_{\theta_1},\ldots,\mathcal{C}^{(s)}=\mathcal{U}_{\theta_s}$, where $(\theta_1,\ldots,\theta_s)\in\overline{\FF_q((X^{-1}))}^s$ is a counterexample to \eqref{eq:Littlewood_series}.  
		It is an interesting problem to determine the $t$-value of the digital sequence.
	
	\end{enumerate}

\end{remark}

As mentioned at the beginning of Section~\ref{sec:4} Sobol' sequences are one of the most popular low-discrepancy sequences used in Quasi-Monte Carlo (QMC) methods. In the theory of uniform distribution one open problem is: what is the best possible order of discrepancy one can obtain for a sequence in $[0,1)^s$. Therefore, it is of interest to get sharp estimates for the discrepancy. Levin established sharp estimates for many known low-discrepancy sequences \cite{LevinH1}, \cite{LevinH2}, \cite{Levints1}, \cite{Levints2} (see also \cite{faureLB} and \cite{HoferPAMS}). So far, sharp estimates for the discrepancy of Kronecker-type sequences for $s>1$ are missing (see \cite{levinLDhybrid} for a related problem). 
\begin{RQ}
	Is it possible to prove sharp discrepancy estimates for a low-discrepancy sequence of the type described in Items~1, 2, or~3 of Remark~\ref{rem:CE_tvalue}? Note that the existence of $\theta_1,\ldots,\theta_s \in\overline{\FF_q((X^{-1}))}$ satisfying \eqref{eq:Littlewood_series_2} remains an open problem.
\end{RQ}
For more details on the definition of the discrepancy and a discussion for sequences connected to several variants of the Littlewood conjecture we refer to \cite{Hofer2026}.

\subsection{Relating counterexample of different variants of the Littlewood conjecture via the generating matrices point of view?}\label{sec:43}
The starting point of the discussion in this section is the following example. 
\begin{example}\label{examp:levin}
	Let $\theta\in\overline{\FF_q((X^{-1}))}$. The matrices $\mathcal{M}_\theta$ and the identity matrix $\mathcal{I}$ in $\FF_q^{\NN_0\times\NN_0}$ generate a digital $(t,2)$-sequence over $\FF_q$ for some $t\in\NN_0$ if and only if $\theta$ is a counterexample to \eqref{eq:padicLittlewood_series} with $p(X)=X$; more precisely, this holds if and only if
	\begin{equation}\label{eq:Littlewood_series_2_CE}
		\inf_{Q\in\FF[X]\setminus\{0\},\,k\geq 0} |Q|\,\cdot\,\|QX^k \theta\|\geq e^{-(t+1)} \mbox{ with }t\in\NN_0.
	\end{equation}
	See \cite[Corollary~2.5, Lemma~2.3, and Lemma~2.4]{Hofer2026}.
\end{example}

Robertson \cite[Lemma~2.0.1]{Robertson2} provided a construction principle that associates, for any irreducible polynomial $P(X)$, a counterexample to \eqref{eq:Littlewood_series_2_CE} with a counterexample to \eqref{eq:padicLittlewood_series}.

\begin{lemma}[\cite{Robertson2}]\label{lem:robertson}
	Let $\theta=\sum_{i=1}^\infty c_iX^{-i}\in\FF_q((X^{-1}))$ and let $t\in\NN_0$ satisfy
	\begin{equation*}
		\inf_{Q\in\FF[X]\setminus\{0\},\,k\geq 0} |Q|\,\cdot\,\|QX^k \theta\|\geq e^{-(t+1)}.
	\end{equation*}
	Let $P(X)\in\FF_q[X]$ be irreducible. Then
	$$\overline{\theta}:=\sum_{i=1}^\infty c_i(P(X))^{-i}$$
	satisfies
	\begin{equation*}
		\inf_{Q\in\FF[X]\setminus\{0\},\,k\geq 0} |Q|\,\cdot\,\|QP(X)^k \overline{\theta}\|\geq e^{-(t+1)\deg(P(X))}.
	\end{equation*}
\end{lemma}

We now reinterpret Lemma~\ref{lem:robertson} from the perspective of generating matrices. The following Corollary~\ref{prop:hankeltransform-p(x)adic} builds on Example~\ref{examp:levin} and analyzes the corresponding generating matrices when replacing $p(X)=X$ with $p(X)=X+a$.

\begin{coro}\label{prop:hankeltransform-p(x)adic}
	Let $\theta\in\overline{\FF_q((X^{-1}))}$ satisfy 
	\begin{equation*}
		\inf_{Q\in\FF[X]\setminus\{0\},k\geq 0} |Q|\,\cdot\,\|QX^k \theta\|\geq e^{-t-1} \mbox{ with }t\in\NN_0.
	\end{equation*} 
	Then the matrices $\mathcal{M}_\theta$ and the identity matrix $\mathcal{I}$ in $\FF_q^{\NN_0\times\NN_0}$ generate a digital $(t,2)$-sequence over $\FF_q$. Let $\mathcal{C}^{(1)}\in\FF_q^{\NN_0\times\NN_0}$ be the matrix constructed by Definition \ref{def:sobol_construction} using $b_1(X)=X+a$, with $a\in\FF$ and $B_1=(1)$.. 
	Then the pair of matrices $\mathcal{M}_\theta\mathcal{C}^{(1)},\mathcal{I}\mathcal{C}^{(1)}$ as well as the pair $(\mathcal{C}^{(1)})^T\mathcal{M}_\theta\mathcal{C}^{(1)},\mathcal{I}\mathcal{C}^{(1)}$  generate a digital $(t,2)$-sequence over $\FF_q$. 

Furthermore $(\mathcal{C}^{(1)})^T\mathcal{M}_\theta\mathcal{C}^{(1)}$ is a Hankel matrix that equals $\mathcal{M}_{\overline{\theta}}$ where 
	  $\overline{\theta}$ is the counterexample to \eqref{eq:padicLittlewood_series} due to Lemma~\ref{lem:robertson} with $p(X)=X+a$, more exactly
	\begin{equation}\label{eq:Littlewood_series_2_CE2}
		\inf_{Q\in\FF[X]\setminus\{0\},k\geq 0} |Q|\,\cdot\,\|Q\cdot (X+a)^k \overline{\theta}\|\geq e^{-t-1}.
	\end{equation}
	
	Moreover, let $\theta=[0;A_1(X),A_2(X),\ldots]$ be the continued fraction of $\theta$ then $$\overline{\theta}=[0;A_1(X+a),A_2(X+a),\ldots].$$
\end{coro}
\begin{proof}
	The first part follows immediately from Example~\ref{examp:levin},  Lemma~\ref{lem:scrambling} and Lemma~\ref{lem:inversematrices}. 
	
	For the second note that the proof of Theorem~\ref{thm:sobol} together with Example~\ref{ex:vdc} and the fact that $\deg(X+a)=1$ shows that both $\mathcal{C}^{(1)}$ and $\mathcal{C}^{(-1)}$
	 can be constructed columnswise, with $c^{(i)}_0=c^{(-i)}_0=(1,0,0,0,\ldots)^T$ and the $r$th columns for $r\in\NN$ obtained via $$c^{(-i)}_r=(\mathcal{X}+a\mathcal{I})^r(1,0,0,0,\ldots)^T\mbox{ and } c^{(i)}_r=(\mathcal{X}-a\mathcal{I})^r(1,0,0,0,\ldots)^T,$$
	 where $\mathcal{X}=\begin{pmatrix}
	 	0\\ \mathcal{I}
	 \end{pmatrix}$. We define $\mathcal{X}^T=\begin{pmatrix}
	 0&\mathcal{I}
	 \end{pmatrix}=:\mathcal{Y}$ and observe for any Hankelmatrix $\mathcal{H}\in\FF_q^{\NN_0\times\NN_0}$ that 
	 \begin{equation}\label{eq:hankel_X_Y}
	 b(\mathcal{Y})\cdot	\mathcal{H}=\mathcal{H}\cdot	b(\mathcal{X})
	 \end{equation}
	 for any polynomial $b(X)\in\FF_q[X]$. 
	 
	 The Hankel property of $(\mathcal{C}^{(1)})^T\mathcal{M}_\theta\mathcal{C}^{(1)}=:\overline{\mathcal{C}}=(c_{k,l})_{k,l\geq 0}$ follows then immediately from 
	 \begin{align*}
	 	c_{k,l}&=(1,0,0,0,\ldots)(\mathcal{Y}-a\mathcal{I})^k \mathcal{M}_\theta (\mathcal{X}-a\mathcal{I})^l (1,0,0,0,\ldots)^T\\
	 	&=(1,0,0,0,\ldots)(\mathcal{Y}-a\mathcal{I})^{k+s} \mathcal{M}_\theta (\mathcal{X}-a\mathcal{I})^{l-s} (1,0,0,0,\ldots)^T=	c_{k+s,l-s}
	 \end{align*}
	 for all $s, l,k\in\NN_0$ with $s\leq l$, where we applied \eqref{eq:hankel_X_Y} in the second equality. 
	 
	 From the construction of the generalized Niederreiter sequences (cf. e.g. \cite{niesiam} or \cite{Tez2}) we know that 
	 that both $\mathcal{C}^{(1)}$ and $\mathcal{C}^{(-1)}$ can be constructed row-wise where the $k$th rows  $\rho_k^{(1)}$ and $\rho_k^{(-1)}$ resp. are determined by the coefficients of the Laurentseries of $\frac{1}{(X+a)^{k+1}}$ and $\frac{1}{(X-a)^{k+1}}$.
	 The first row of $\mathcal{M}_\theta$ is determined by the coefficients of $\theta=\theta(X)=\sum_{i=1}^{\infty}a_iX^{-i}$. Therefore we see the first row of $(\mathcal{C}^{(1)})^T\mathcal{M}_\theta\mathcal{C}^{(1)}=:\mathcal{M}_{\overline{\theta}}$ is determined by the coefficients of $$\overline{\theta}=\overline{\theta}(X)=\sum_{i=1}^{\infty}a_i(X+a)^{-i}=\theta(X+a).$$ 
	 This together with Lemma~\ref{lem:robertson} ensures \eqref{eq:Littlewood_series_2_CE2} and the statement of the continued fraction of $\overline{\theta}$. 
\end{proof}	
The part of the proof of Corollary~\ref{prop:hankeltransform-p(x)adic} that relates $\theta(X)$ with $\overline{\theta}(X)$ together with the digital point of view in Example~\ref{examp:levin} allows a slight extension of Lemma~\ref{lem:robertson}. 
 
\begin{remark}\label{rem:5}
		Note that Corollary~\ref{prop:hankeltransform-p(x)adic} only covers the case where $p(X)$ in Lemma~\ref{lem:robertson} is a linear polynomial $p(X)=X+a$. The identity $(\mathcal{C}^{(1)})^T\mathcal{M}_\theta\mathcal{C}^{(1)}=\mathcal{M}_{\overline{\theta}}$ can also be explained via a binomial transforms of Hankel matrices. 

In the general case of any irreducible $p(X)$ we can only link the first row of $M_{\overline{\theta}}$ due to the counterexample in Lemma~\ref{lem:robertson} with the first row of a matrix $M_{{\theta}}\mathcal{P}$, where the $k$th row  $\rho_k$ of $\mathcal{P}$, $k\in\NN_0$, is determined by the coefficients of the Laurent series of $\frac{1}{(p(X))^{k+1}}$. 

	Let $\theta=\sum_{i=1}^\infty c_iX^{-i}\in\FF_q((X^{-1}))$ and $t \in\NN_0$ such that \begin{equation*}
	\inf_{Q\in\FF[X]\setminus\{0\},k\geq 0} |Q|\,\cdot\,\|QX^k \theta\|\geq e^{-(t+1)} \mbox{ with }t\in\NN_0.
\end{equation*}
Since for any $l\in\NN_0$, trivially $\theta^{(l)}=\sum_{i=1}^\infty c_{i+l}X^{-i}\in\FF_q((X^{-1}))$ satisfies
\begin{equation*}
	\inf_{Q\in\FF[X]\setminus\{0\},k\geq 0} |Q|\,\cdot\,\|QX^{k} \theta^{(l)}\|\geq	\inf_{Q\in\FF[X]\setminus\{0\},k\geq 0} |Q|\,\cdot\,\|QX^k \theta\|\geq e^{-(t+1)}.
\end{equation*}
We know any $\overline{\theta}^{(l)}:=\sum_{i=1}^\infty c_{i+l}(P(X))^{-i}$ satisfies \begin{equation*}
	\inf_{Q\in\FF[X]\setminus\{0\},k\geq 0} |Q|\,\cdot\,\|QP(X)^k \overline{\theta}^{(l)}\|\geq e^{-(t+1)\deg(P(X))}.
\end{equation*}
Therefore not only the first row in $M_{{\theta}}\mathcal{P}$ but any arbitrary row can be identified with a counterexample to \eqref{eq:padicLittlewood_series}.
	
\end{remark}

An interesting research question is the following that uses the notation in Remark~\ref{rem:5}:
\begin{RQ}
 Is it possible to find a matrix $\mathcal{Q}$ such that 
	$$\mathcal{Q}M_{{\theta}}\mathcal{P}=M_{\overline{\theta}}.$$ Note that this would generalize in a certain sense the binomial transform of a Hankel matrix. 
\end{RQ}	
Where does a generalization of the ideas of proof of Proposition \ref{prop:hankeltransform-p(x)adic} fail when switching from a linear polynomial $b_1(X)=X+a$ to a polynomial $b_1(X)$ with $\deg(b_1(X))=e>1$?
	Note that from the proof of Theorem~\ref{thm:sobol} we can derive the following columnwise construction of $\mathcal{C}^{(-1)}$ for any $b_1(X)\in\FF_q[X]$ with $\deg(b_1(X))=e\geq 1$ and $B_1$ the $e\times e$ identity matrix. We denote the $j$th column of $\mathcal{C}^{(-1)}$ as $c^{(-1)}_{j}$. We write $j=e\cdot k+l$ with $k,l\geq 0$ and $l\in\{0,1,\ldots,e-1\}$. Then
	
	$$c^{(-1)}_{e k+l}=b_1(\mathcal{X})^k\mathcal{X}^l (1,0,0,0,\cdots)^T.$$ 

	Now completely analogously as in the proof of Proposition~\ref{prop:hankeltransform-p(x)adic} we obtain for any Hankel matrix $\mathcal{H}$ via
	$$\mathcal{H}b_1(\mathcal{X})^k=b_1(\mathcal{Y})^k\mathcal{H}.$$
	Thus, we do indeed get that that $(\mathcal{C}^{(-1)})^T\mathcal{H}\mathcal{C}^{(-1)}=:\overline{\mathcal{H}}=(\overline{h}_{i,j})_{i,j\geq 0}$ is a \emph{$e\times e$ pixel-Hankel matrix}, that satisfies $\overline{h}_{i,j}=\overline{h}_{i+e\cdot s,j-e \cdot s}$ for all  $i,j,s\geq 0$ such that $j-e \cdot s\geq 0$. But whenever $e>1$ it is not a Hankel matrix in the strict sense. 
	
	A second approach is to construct a matrix $\mathcal{T}$ with $j$th column $t_j$ defined by	$$t_{j}=b_1(\mathcal{X})^j t_0$$	with $t_{0}=(1,0,0,0,\cdots)^T $or any other initial setting. 
	Then $(\mathcal{T})^T\mathcal{H}\mathcal{T}$ remains a Hankel matrix, but the rank-structure is not preserved.

In the theory of uniform distribution one is interested in finding new examples of generating matrices for $(t,s)$-sequences in the sense of Niederreiter. Example~\ref{examp:levin} was first suggested by Levin~\cite{levinLDhybrid}. 
We would like to point out that Example~\ref{examp:levin} shows for any counterexample $\theta$ satisfying \eqref{eq:Littlewood_series_2_CE}, that the two matrices $\mathcal{M}_\theta$ and $\mathcal{I}$ construct a $(t,2)$-sequence.  Using Theorem~\ref{prop:1}, Lemma~\ref{lem:inver_U_theta}, together with Lemma~\ref{lem:scrambling} we see that also the two matrices $\mathcal{I}_\theta$ and $\mathcal{V}_\theta$ due to Lemma~\ref{lem:inver_U_theta} construct a $(t,2)$-sequence. This observation raises the following research question: 

\begin{RQ}
	An interesting problem is to determine, for a given NUT matrix $\mathcal{V}\in\FF_q^{\NN_0\times\NN_0}$ such that $\mathcal{V}$ and $\mathcal{I}$ generate a $(t,s)$-sequence over $\FF_q$, sufficient or necessary conditions under which $\mathcal{V}=\mathcal{V}_\theta$ for some $\theta\in\FF_q((X^{-1}))$.
\end{RQ}
The question above is related to describing the set of all possible counterexamples to the $X$-adic Littlewood conjecture.
A natural question is whether this perspective can provide new insights into the $X$-adic Littlewood conjecture over $\FF_2$, which remains unresolved. In particular, can the study of generating matrices yield information in this case? We note that NUT generating matrices of $(0,2)$-sequences over $\FF_2$ have been completely characterized; see \cite[Proposition~4]{HoferLarcher2010}.

One possible initial approach to \eqref{eq:Littlewood_series} is to start with a counterexample $\theta$ to the $X$-adic Littlewood conjecture and ask whether it is possible to construct from it a counterexample to the Littlewood conjecture over $\FF_q((X^{-1}))$.
In other words, one may ask the following.

\begin{RQ}
	Is it possible to find $\theta,\phi\in\FF_q((X^{-1}))$ such that $\theta$ satisfies \eqref{eq:Littlewood_series_2_CE} and there exists an NLT matrix $\mathcal{Q}$, or a related transformation, for which
	\[
	\mathcal{Q}\mathcal{M}_{\theta}\mathcal{U}_{\phi}
	\]
	is again a Hankel matrix $\mathcal{M}_{\xi}$? If so, then $(\xi,\phi)$ would be a candidate for a counterexample to \eqref{eq:Littlewood_series}.
\end{RQ}	

For example, suppose that $\phi=[0;X,X,X,\ldots]$. Then, by Definition~\ref{def:mat_RandI}, we have
\[
\mathcal{I}_\phi=\mathcal{I}
\qquad\text{and}\qquad
\mathcal{R}_\phi=\overline{\mathcal{R}}_\phi=\mathcal{X}-\mathcal{Y}.
\]
Consequently, Definition~\ref{def:mat_UandL} yields a columnwise construction of $\mathcal{U}_\phi$, whose $j$th column is given by
\[
u_j=(\mathcal{X}-\mathcal{Y})^j(1,0,0,0,\ldots)^{\mathsf T}.
\]
Unfortunately, although
\[
\mathcal{Y}\mathcal{H}=\mathcal{H}\mathcal{X}
\]
holds for every Hankel matrix $\mathcal{H}$, the relation
\[
\mathcal{X}\mathcal{H}=\mathcal{H}\mathcal{Y}
\]
does not hold for any nonzero Hankel matrix $\mathcal{H}$. Therefore,
\[
\mathcal{U}_\phi^{\mathsf T}\mathcal{H}\mathcal{U}_\phi
\]
is symmetric, but it is not necessarily a Hankel matrix.

\section{AI Declaration}
The mathematical content of this work was developed without the use of artificial intelligence tools and is solely the work of the author. Artificial intelligence tools were used only to help present proof ideas and polish the language.

\section*{Appendix}

In what follows, we list some analogies between $\RR$ and $\FF((X^{-1}))$ that are useful in this paper.  
\begin{enumerate}

	\item Every real number $\alpha \in [0,1)$ can be expressed as a continued fraction $[0;a_1,a_2,a_3,\ldots]$ with $a_i\in\NN$ for every $i\in\NN$. \\
	
	Analogously, every series $L\in\overline{\FF((X^{-1}))}$ can be expressed as a continued fraction 
	$$[0;A_1(X),A_2(X),A_3(X),\ldots]=\cfrac{1}{A_1(X)+\cfrac{1}{A_2(X)+\cfrac{1}{A_3(X)+\ddots} }}$$
	with $A_i(X)\in\FF[X]$ satisfying $\deg(A_i(X))\geq 1$ for every $i\in\NN$.
	

\item The $n$th convergent of a non-rational $\alpha\in[0,1)$ is the finite continued fraction $[0;a_1,a_2,\ldots,a_n]$, which can be written as $p_n/q_n$ with $p_n,q_n \in\NN_0$. The sequences $(p_n)_{n\geq 0}$ and $(q_n)_{n\geq 0}$ are computed recursively according to
$q_n=a_nq_{n-1}+q_{n-2}$ and $p_n=a_np_{n-1}+p_{n-2}$ for every $n\in\NN$, with initial values $q_{-1}=0,\,p_{-1}=1,\,q_0=1$, and $p_0=0$.  \\
Analogously, the $n$th convergent of a non-rational $L\in\overline{\FF((X^{-1}))}$ is the finite continued fraction $[0;A_1(X),A_2(X),\ldots,A_n(X)]$, which can be written as $P_n(X)/Q_n(X)$ with $P_n(X),Q_n(X) \in\FF[X]$. The sequences $(P_n(X))_{n\geq 0}$ and $(Q_n(X))_{n\geq 0}$ are computed recursively according to
$Q_n(X)=A_n(X)Q_{n-1}(X)+Q_{n-2}(X)$ and $P_n(X)=A_n(X)P_{n-1}(X)+P_{n-2}(X)$ for every $n\in\NN$, with initial values $Q_{-1}(X)=0,\,P_{-1}(X)=1,\,Q_0(X)=1$, and $P_0(X)=0$.  

\item An interesting fact is that the denominators $q_n$ and $Q_n(X)$ provide best approximations in the following sense: if 
$$|\alpha-p/q|<1/|q|^2,$$
then $p/q$ is a convergent to $\alpha$, and if 
$$\nu(L-P(X)/Q(X))<2\nu\left( \frac{1}{Q(X)}\right)=-2 \deg(Q(X)),$$
then $P(X)/Q(X)$ is a convergent to $L$. 

Moreover, we have $\deg(Q_n)=\sum_{i=1}^n\deg(A_i)=:d_n$, $\nu(L-P_n(X)/Q_n(X))=-d_n-d_{n+1}$, and
$$\nu(L-v(X)/w(X))\geq \nu(L-P_n(X)/Q_n(X))$$
for all $v(X),w(X)\in\FF[X]$ satisfying $0\leq \deg(w(X))< d_{n+1}$.


\item Finally, we point out two further analogies between $\RR$ and the set of formal Laurent series $\FF((X^{-1}))$ when $\FF$ is a finite field. 
We have the following two classifications of $\alpha\in\RR$:
\begin{itemize}
	\item[-] $\alpha$ is a rational number if and only if, for every integer $b>1$, the $b$-adic expansion of $\alpha$ is eventually periodic or finite. 
	\item[-] $\alpha$ is a quadratic irrational number if and only if its continued-fraction expansion is infinite and ultimately periodic. 
\end{itemize}
When the field $\FF$ is finite, the condition $L\in\FF(X)$ for $L\in\FF((X^{-1}))$ is equivalent to the condition that the series expansion $\sum_{i=u}^\infty c_iX^{-i}$ is eventually periodic or finite. (See \cite[Theorem~1.1]{lasj2000}.)

Moreover, if $\FF$ is finite and $L\in\FF((X^{-1}))\setminus \FF(X)$, then the continued-fraction expansion of $L$ is ultimately periodic if and only if $L$ is quadratic over $\FF(X)$ (cf. e.g. \cite[Theorem~3.1]{lasj2000}). 

For results and discussions in the case where $\FF$ is not finite, see, for example, \cite{schmidt2000}. 
\end{enumerate}

For further information on continued fractions in power-series fields, we refer the reader to \cite{lasj2000} and \cite{schmidt2000}.

\begin{small}
\noindent{\scshape{Authors' Address:}}\\

\noindent Roswitha Hofer:\\
Institute of Financial Mathematics and Applied Number Theory,\\
University of Linz\\
Altenbergerstr.~69\\
A-4040 Linz\\
Austria\\
E-mail: \texttt{roswitha.hofer@jku.at}
\end{small}

\end{document}